\input{hyperref-usetitle.sty}
\input{cleveref-options.sty}
\documentclass{article}
\pdfoutput=1

\usepackage{biblatex}
\usepackage{xr}
\usepackage{subfiles}

\usepackage{biblatex-options}
\usepackage{my-preamble}

\usepackage{amsthm}
\usepackage{cleveref}
\newtheorem{thmcounter}{Dummy}[section]
\theoremstyle{definition}
\newtheorem{definition}[thmcounter]{Definition}
\newtheorem{construction}[thmcounter]{Construction}
\Crefname{construction}{Construction}{Constructions}

\theoremstyle{plain}
\newtheorem{axiom}{Axiom}[section]
\Crefname{axiom}{Axiom}{Axioms}

\newtheorem{task}{Task}
\Crefname{task}{Task}{Tasks}

\newtheorem{proposition}[thmcounter]{Proposition}
\newtheorem{theorem}[thmcounter]{Theorem}
\newtheorem{lemma}[thmcounter]{Lemma}
\newtheorem{corollary}[thmcounter]{Corollary}
\newtheorem{problem}[thmcounter]{Problem}

\theoremstyle{remark}
\newtheorem{remark}[thmcounter]{Remark}
\newtheorem{example}[thmcounter]{Example}

\newtheorem{warning}[thmcounter]{Warning}

\title{Normalization and coherence for $\infty$-type theories}
\author{Taichi Uemura}

\begin{document}

\maketitle

\begin{abstract}
  We develop a technique for normalization for $\infty$-type theories. The
normalization property helps us to prove a coherence theorem: the
initial model of a given $\infty$-type theory is $0$-truncated. The
coherence theorem justifies interpreting an ordinary type theory in
$(\infty, 1)$-categorical structures.

\end{abstract}

\section{Introduction}
\label{sec:introduction}

\emph{\(\infty\)-type theories} \parencite{nguyen2022type-arxiv} are an
\(\infty\)-dimensional generalization of dependent type theories and
introduced by \citeauthor{nguyen2022type-arxiv} to tackle
\emph{coherence problems} in the \((\infty, 1)\)-categorical semantics of
type theories. Motivated by the development of homotopy type theory
\parencite{hottbook}, certain type theories are conjectured to admit
interpretations in structured \((\infty, 1)\)-categories
\parencite{kapulkin2018homotopy,kapulkin2017locally}, but one
immediately faces a coherence problem, namely a mismatch between
strictness levels of equality. Strict equality in type theories has to
be interpreted as a much weaker notion of equality, homotopies, in
\((\infty, 1)\)-categories. \(\infty\)-type theories are kind of type theories
in which equations are up to homotopy so that there is no coherence
problem between \(\infty\)-type theories and structured
\((\infty, 1)\)-categories. The strategy for the coherence problem between
a type theory and an \((\infty, 1)\)-categorical structure is to introduce
an intermediate \(\infty\)-type theory and reduce the problem to a
coherence problem between the type theory and the \(\infty\)-type
theory. Then the coherence problem is essentially whether the initial
model (and free models, more generally) of the \(\infty\)-type theory is
\(0\)-truncated. However, a solution to the coherence problem is only
given for a specific type theory \parencite{nguyen2022type-arxiv}.

\emph{Normalization} for a type theory is the property that every type
or term in the type theory has a unique \emph{normal form}. Because
the normal forms are inductively defined without quotiented out by any
relation, normalization provides a way to calculate equality of types
and terms. A consequence is for example the decidability of equality
of types and terms.

In this paper, we develop a technique for \emph{normalization for
  \(\infty\)-type theories} as a solution to coherence problems at some
level of generality. Here, the unique existence of normal forms is
understood in a higher dimensional sense: the space of normal forms of
every type or term is contractible. Normalization helps us to
calculate the path spaces of the initial model of an \(\infty\)-type theory
and to determine the truncation level of the initial model. The
difficulty is that \(\infty\)-type theories are \emph{both} higher
dimensional and dependently typed.

Normalization for a \(2\)-dimensional simple type theory is proved by
\textcite{fiore2020coherence}. They also derive from it a coherence
theorem for cartesian closed bicategories. Their proof is purely
categorical and seems to scale to the \(\infty\)-dimensional setting, but
it does not work directly for dependent type theories. For a simple
type theory, only terms are normalized. Normalization for a dependent
type theory is more complicated because types and terms are
simultaneously normalized.

A successful approach to normalization for \(1\)-dimensional dependent
type theories is \citeauthor{sterling2021thesis}'s \emph{synthetic
  Tait computability}
\parencite{sterling2021thesis,sterling2021normalization,gratzer2021normalization}. It
is a technique of constructing \emph{logical relations} using an
internal language of a category obtained by the \emph{Artin
  gluing}. We adapt it to the higher dimensional setting using
\emph{homotopy type theory} \parencite{hottbook} as an internal
language of an \((\infty, 1)\)-category obtained by the Artin gluing.

The logical relation constructed using synthetic Tait computability
yields a \emph{normalization function} that assigns a normal form to
each type or term. The \emph{uniqueness} of normal forms is, however,
proved by induction \emph{externally} to the glued category in prior
work
\parencite{sterling2021thesis,sterling2021normalization,gratzer2021normalization}. This
no longer works in the higher dimensional setting because, once
externalized, some construction involves infinitely many coherence
data difficult to handle. An internal approach is taken by
\textcite{bocquet2021relative} who prove the uniqueness of normal
forms within the framework of multimodal type theory
\parencite{gratzer2021multimodal}, but its interpretation in
\((\infty, 1)\)-categories is not clear. We instead embed the glued
\((\infty, 1)\)-category into another \((\infty, 1)\)-category to which the
normalization function is internalized and then use ordinary homotopy
type theory as its internal language. This is achieved by the
technique developed by \textcite{uemura2022diagrams-arxiv} of
internalizing a diagram of \((\infty, 1)\)-categories to an
\((\infty, 1)\)-category obtained by a generalization of the Artin gluing.

\subsection{Contributions}
\label{sec:contributions}

Our contribution is a general strategy for normalization and coherence
for \(\infty\)-type theories. This is a combination of
\citeauthor{fiore2020coherence}'s idea of deriving coherence from
normalization \parencite{fiore2020coherence} and
\citeauthor{sterling2021thesis}'s synthetic Tait computability for
normalization for dependent type theories
\parencite{sterling2021thesis,sterling2021normalization,gratzer2021normalization}. We
demonstrate our strategy by a simple example of \(\infty\)-type
theory. Heuristically, our strategy works for many fragments and
variants of (an \(\infty\)-analogue of) Martin-L{\"o}f type theory
\parencite{martin-lof1975intuitionistic,nordstrom1990programming}, and
thus interpretations of such type theories in structured
\((\infty, 1)\)-categories are justified. Finding a good sufficient
condition for an \(\infty\)-type theory to admit normalization and
coherence is left as future work.

The core idea of our normalization proof is a translation of prior
work, especially
\parencite{sterling2021thesis,sterling2021normalization,gratzer2021normalization,bocquet2021relative},
into a higher dimensional language. A technical novelty is the use of
the internalization technique of \textcite{uemura2022diagrams-arxiv}
to prove the uniqueness of normal forms within ordinary homotopy type
theory.

\subsection{Organization}
\label{sec:organization}

In \cref{sec:infty-type-theories}, we review \(\infty\)-type theories
\parencite{nguyen2022type-arxiv}. In \cref{sec:mode-sketches}, we
review the technique of internalizing a diagram of
\((\infty, 1)\)-categories to homotopy type theory
\parencite{uemura2022diagrams-arxiv}. A higher dimensional version of
synthetic Tait computability is obtained as a special case.

From \cref{sec:normal-forms-1}, we begin to demonstrate our
normalization technique and coherence results. In
\cref{sec:normal-forms-1} we construct the type of normal forms
internally to a certain \((\infty, 1)\)-category. In
\cref{sec:normalization-models}, we build the logical relation for
normalization internally to the \((\infty, 1)\)-category used in
\cref{sec:normal-forms-1}. For this we essentially follow prior work
\parencite{sterling2021thesis,sterling2021normalization,gratzer2021normalization},
but the presentation might look different because of the difference of
internal languages to be used. \Cref{sec:norm-results} is devoted to
the uniqueness of normal forms. We essentially follow
\textcite[Section D.5]{bocquet2021relative}. While they use the
framework of multimodal type theory \parencite{gratzer2021multimodal},
our proof is internal to homotopy type theory. Finally in
\cref{sec:coherence-theorem}, we prove coherence results. They depend
on the results in \cref{sec:norm-results} but do not depend on any
intermediate concept introduced in
\cref{sec:normalization-models,sec:norm-results}. Therefore, the
reader may jump from \cref{sec:normal-forms-1} to
\cref{sec:coherence-theorem}, admitting the normalization results.

\Cref{sec:proof-relative-induction,sec:proof-proposition} contain
proofs of two technical propositions stated in
\cref{sec:norm-results,sec:coherence-theorem}.

\subsection{Preliminaries}
\label{sec:preliminaries}

We extensively use \emph{homotopy type theory} \parencite{hottbook} as
internal languages of sufficiently rich \((\infty, 1)\)-categories. The
reader is assumed to be familiar with homotopy type theory and its
semantics
\parencite{awodey2009homotopy,kapulkin2021simplicial,shulman2015inverse,shulman2019toposes,arndt2011homotopy}. Familiarity
with \((\infty, 1)\)-category theory
\parencite{lurie2009higher,cisinski2019higher,riehl2022elements} is
helpful but not necessary. We use a high level language of
\((\infty, 1)\)-category theory which can be understood by analogy with
ordinary category theory. We sometimes use the language of
\((\infty, 2)\)-category theory but do not dive into details.
\(\Fun(\cat, \catI)\) denotes the \((\infty, 1)\)-category (or
\((\infty, 2)\)-category) of functors \(\cat \to \catI\), and
\(\neord[\nat]\) denotes the poset \(\{0 < 1 < \dots < \nat\}\).

\subsection{Related work}
\label{sec:related-work}

Our normalization proof is a translation of \emph{normalization by
  evaluation} into a higher dimensional language. It has first been
used by \textcite{martin-lof1975intuitionistic} and independently by
\textcite{berger1991inverse}. \Textcite{altenkirch1995categorical,fiore2002semantics,fiore2022semantic-arxiv,coquand2019canonicity}
give categorical and algebraic accounts of normalization by
evaluation. \Textcite{sterling2021thesis} simplifies the construction
of logical relations using synthetic Tait computability, leading
normalization for complex dependent type theories
\parencite{sterling2021normalization,gratzer2021normalization}. \Textcite{bocquet2021relative}
prove the uniqueness of normal forms internally to multimodal type
theory
\parencite{gratzer2021multimodal}. \Citeauthor{fiore2020coherence}'s
work \parencite{fiore2020coherence} is the first step towards higher
dimensional normalization by evaluation. Our work is along these
lines.

Proving coherence theorems via normalization is not
new. \Citeauthor{maclane1963natural}'s coherence theorem for monoidal
categories \parencite{maclane1963natural} is proved by rewriting every
term in the language of monoids to a normal form. The crucial part is
that any parallel two rewriting paths are equal, which implies the
uniqueness of normal forms in the \(2\)-dimensional
sense. \Textcite{curien1993substitution} uses normalization by
rewriting to solve the coherence problem in
\citeauthor{seely1984locally}'s interpretation of dependent type
theory in locally cartesian closed categories
\parencite{seely1984locally}. \Textcite{fiore2020coherence} use
normalization by evaluation to prove a coherence theorem for cartesian
closed bicategories.

The approach to coherence problems in the previous work
\parencite{nguyen2022type-arxiv} is a strictification technique along
the lines of
\textcite{hofmann1995interpretation,lumsdaine2015local}. This approach
does not work well in the \((\infty, 1)\)-categorical semantics, because
homotopies between morphisms have to be strictified unlike the
\(1\)-categorical case. \Textcite{cherradi2022interpreting-arxiv}
makes progress in strictifying type-theoretic structures on
\((\infty, 1)\)-categories such as function types and inductive types.

\section{$\infty$-type theories}
\label{sec:infty-type-theories}

\(\infty\)-type theories \parencite{nguyen2022type-arxiv} generalize the
categorical definition of type theories given by
\textcite{uemura2019framework}. Recall that a morphism
\(\mor : \objI \to \obj\) in an \((\infty, 1)\)-category \(\cat\) with finite
limits is \emph{exponentiable} if the pullback functor
\(\mor^{\pbMark} : \cat_{\slice \obj} \to \cat_{\slice \objI}\) has a
right adjoint \(\mor_{\pbMark}\) called the \emph{pushforward along
  \(\mor\)}.

\begin{definition}
  \label{def:infinity-cwr}
  An \emph{\((\infty, 1)\)-category with representable maps} is an
  \((\infty, 1)\)-category \(\cat\) with finite limits equipped with
  a class of exponentiable morphisms in \(\cat\) called
  \emph{representable maps} closed under identities, composition,
  and pullbacks along arbitrary morphisms. A \emph{morphism of
  \((\infty, 1)\)-categories with representable maps} is a
  functor preserving finite limits, representable maps, and
  pushforwards along representable maps.
\end{definition}

Recall that the right fibrations over an \((\infty, 1)\)-category
\(\cat\) are equivalent to the presheaves on \(\cat\)
\parencite[Chapter 5]{cisinski2019higher}.

\begin{example}
  \label{exm:cwr-diagram-rfib}
  Let \(\Cat^{\nMark{2}}\) denote the \((\infty, 2)\)-category of
  \((\infty, 1)\)-categories. Let
  \(\cat : \idxsh \to \Cat^{\nMark{2}}\) be a functor from an
  \((\infty, 2)\)-category \(\idxsh\). Let
  \(\Fun(\idxsh, \RFib)_{\cat}\) be the full subcategory of
  \(\Fun(\idxsh, \Cat^{\nMark{2}})_{\slice \cat}\) spanned by those
  objects \(\sh\) whose components \(\sh_{\idx} \to \cat_{\idx}\) for
  \(\idx \in \idxsh\) are all right fibrations.
  \(\Fun(\idxsh, \RFib)_{\cat}\) has a canonical structure of
  \((\infty, 1)\)-category with representable maps where a map
  \(\map : \shI \to \sh\) is representable if the functor
  \(\map_{\idx} : \shI_{\idx} \to \sh_{\idx}\) has a right adjoint for
  every object \(\idx \in \idxsh\) and if the naturality square
  \begin{equation*}
    \begin{tikzcd}[row sep = 3ex]
      \shI_{\idx}
      \arrow[r, "\shI_{\idxmor}"]
      \arrow[d, "\map_{\idx}"'] &
      \shI_{\idxI}
      \arrow[d, "\map_{\idxI}"] \\
      \sh_{\idx}
      \arrow[r, "\sh_{\idxmor}"'] &
      \sh_{\idxI}
    \end{tikzcd}
  \end{equation*}
  is a morphism of adjunctions from \(\map_{\idx}\) to
  \(\map_{\idxI}\) for any morphism \(\idxmor : \idx \to \idxI\). The
  pushforward \(\map_{\pbMark} \shII\) for \(\shII \to \shI\) is given
  by the pullback
  \((\map_{\pbMark} \shII)_{\idx} \simeq \mapI_{\idx}^{\pbMark}
  \shII_{\idx}\) where \(\mapI_{\idx}\) is the right adjoint of
  \(\map_{\idx}\). It then follows that the precomposition with an
  arbitrary functor \(\fun : \idxsh' \to \idxsh\) induces a morphism of
  \((\infty, 1)\)-categories with representable maps
  \(\Fun(\idxsh, \RFib)_{\cat} \to \Fun(\idxsh', \RFib)_{\cat \comp
    \fun}\). When \(\idxsh = \neord[0]\), we simply write
  \(\RFib_{\cat} = \Fun(\neord[0], \RFib)_{\cat}\).
\end{example}

\begin{definition}
  \label{def:infinity-type-theory}
  An \emph{\(\infty\)-type theory} is a small \((\infty, 1)\)-category with
  representable maps. A \emph{\(1\)-type theory} is an \(\infty\)-type
  theory whose underlying \((\infty, 1)\)-category is a
  \((1, 1)\)-category.

  \label{def:model-of-infinity-type-theory}
  Let \(\tth\) be an \(\infty\)-type theory. A \emph{model \(\model\) of
    \(\tth\)} consists of an \((\infty, 1)\)-category \(\Ctx(\model)\) with
  a terminal object and a morphism of \((\infty, 1)\)-categories with
  representable maps \(\model : \tth \to \RFib_{\Ctx(\model)}\). For an
  object \(\obj \in \tth\), we sometimes write \(\obj(\model)\) instead
  of \(\model(\obj)\). In other words, we identify \(\obj\) with the
  map sending a model \(\model\) to its ``\(\obj\)-component''
  \(\model(\obj)\). A \emph{morphism \(\model \to \modelI\) of models of
    \(\tth\)} consists of a functor
  \(\fun_{\Ctx} : \Ctx(\model) \to \Ctx(\modelI)\) preserving terminal
  objects and a morphism
  \(\fun : \tth \to \Fun(\neord[1], \RFib)_{\fun_{\Ctx}}\) of
  \((\infty, 1)\)-categories with representable maps whose domain and
  codomain projections are \(\model\) and \(\modelI\), respectively. A
  model \(\model\) of \(\tth\) is a \emph{\(1\)-model} if
  \(\Ctx(\model)\) is a \((1, 1)\)-category and if
  \(\model : \tth \to \RFib_{\Ctx(\model)}\) is valued in
  \(0\)-truncated objects.
\end{definition}

Since the definition of \(\infty\)-type theories is \emph{essentially
algebraic} (more precisely, the \((\infty, 1)\)-category of \(\infty\)-type
theories is \emph{presentable}
\parencite[Section 3.1]{nguyen2022type-arxiv}), examples of \(\infty\)-type
theories are presented by \emph{generators and relations}.

\begin{example}
  \label{exm:basic-dependent-type-theory}
  The \(\infty\)-type theory \(\basicTT\) freely generated by one
  representable map \(\typeof : \Tm \to \Ty\) exists. The space of
  morphisms of \(\infty\)-type theories \(\basicTT \to \tth\) is equivalent
  to the space of representable maps in \(\tth\). Therefore, a model
  \(\model\) of \(\basicTT\) consists of an \((\infty, 1)\)-category
  \(\Ctx(\model)\) with a terminal object and a representable map
  \(\typeof(\model) : \Tm(\model) \to \Ty(\model)\) of right fibrations
  over \(\Ctx(\model)\). This is a higher dimensional generalization
  of the notion of a \emph{natural model}
  \parencite{awodey2018natural,fiore2012discrete}. We think of
  \(\Ctx(\model)\) as the \((\infty, 1)\)-category of \emph{contexts} and
  \emph{substitutions} in \(\model\), \(\Ty(\model)\) as the right
  fibration of \emph{types} in \(\model\), and \(\Tm(\model)\) as the
  right fibration of \emph{terms} in \(\model\).
\end{example}

\begin{example}
  \label{exm:type-theory-truncation}
  Let \(\tth\) be an \(\infty\)-type theory. The universal \(1\)-type
  theory \(\trunc{\tth}_{1}\) under \(\tth\) exists. It is a
  \(1\)-type theory equipped with a morphism
  \(\tth \to \trunc{\tth}_{1}\) along which any morphism
  \(\tth \to \tthI\) to any \(1\)-type theory \(\tthI\) uniquely
  extends.
\end{example}

\subsection{Coherence problems}
\label{sec:coherence-problems}

Recall that any \(\infty\)-type theory has the initial model
\parencite[Section 4.2]{nguyen2022type-arxiv}.

\begin{problem}[Coherence problem]
  \label{prb:coherence-problem}
  Let \(\tth\) be an \(\infty\)-type theory. Does the initial model of
  \(\tth\) coincide with the initial model of \(\trunc{\tth}_{1}\)
  (\cref{exm:type-theory-truncation})?  Equivalently, is the initial
  model of \(\tth\) a \(1\)-model?
\end{problem}

\(\trunc{\tth}_{1}\) is to be a type theory we are interested in. A
model of \(\tth\) is considered as a \emph{higher-dimensional model}
of \(\trunc{\tth}_{1}\). Higher-dimensional models are easily
constructed from \((\infty, 1)\)-categories with some structures. Once the
coherence problem is positively solved, the initial model of
\(\trunc{\tth}_{1}\) admits a unique morphism to an arbitrary model of
\(\tth\), which justifies interpreting the type theory
\(\trunc{\tth}_{1}\) in a higher structure.

The notion of \(\infty\)-type theory is too general to solve
\cref{prb:coherence-problem} at once. We restrict our attention to
\emph{cellular} \(\infty\)-type theories in the sense that they are
obtained from \(\basicTT\) (\cref{exm:basic-dependent-type-theory}) by
freely adjoining morphisms and homotopies. This notion covers
Martin-L{\"o}f type theory \parencite{martin-lof1975intuitionistic}
and its extension by new type constructors, but does not cover type
theories with new judgment forms such as cubical type theory
\parencite{cohen2016cubical}.

\subsection{Interpretations in $\infty$-logoses}
\label{sec:interpr-infty-toposes}

Recall that an \emph{\(\infty\)-logos}, also known as an
\emph{\(\infty\)-topos} \parencite{lurie2009higher,anel2021topo-logie}, is
an \((\infty, 1)\)-category which admits an interpretation of homotopy type
theory \parencite{hottbook}. The interpretation is justified by
strictifying the \(\infty\)-logos \parencite{shulman2019toposes} and then
by the categorical semantics of homotopy type theory
\parencite{kapulkin2021simplicial,lumsdaine2015local,awodey2009homotopy,shulman2015inverse}.

\begin{definition}
  Let \(\tth\) be an \(\infty\)-type theory and \(\logos\) an
  \(\infty\)-logos. An \emph{interpretation of \(\tth\) in \(\logos\)} is a
  morphism of \((\infty, 1)\)-categories with representable maps
  \(\tth \to \logos\) where we choose all the maps in \(\logos\) as
  representable maps.
\end{definition}

When \(\tth\) is presented by some data, interpretations of \(\tth\)
in an \(\infty\)-logos \(\logos\) can be axiomatized in the internal
language of \(\logos\).

\begin{definition}
  \label{def:type-term-structure}
  In type theory, a \emph{type-term structure} \(\tytmstr\) consists
  of the following data.
  \begin{align*}
    \begin{autobreak}
      \MoveEqLeft
      \ilTy_{\tytmstr} :
      \ilUniv
    \end{autobreak}
    \\
    \begin{autobreak}
      \MoveEqLeft
      \ilTm_{\tytmstr} :
      \ilTy_{\tytmstr}
      \to \ilUniv
    \end{autobreak}
  \end{align*}
  We regard \(\tytmstr \mapsto \ilTy_{\tytmstr}\) as an implicit coercion
  from type-term structures to types and \(\ilTm_{\tytmstr}\) as an
  implicit coercion from \(\ilTy_{\tytmstr}\) to types. That is, we
  write \(\ty : \tytmstr\) to mean \(\ty : \ilTy_{\tytmstr}\) and
  \(\tm : \ty\) for \(\ty : \tytmstr\) to mean
  \(\tm : \ilTm_{\tytmstr}(\ty)\).
\end{definition}

\begin{example}
  \label{exm:interpretation-D}
  An interpretation of \(\basicTT\) in \(\logos\) is a type-term
  structure in the internal language of \(\logos\). An interpretation
  of a cellular \(\infty\)-type theory \(\tth\) in \(\logos\) is a
  type-term structure equipped with some operators and homotopies
  which we call a \emph{\(\tth\)-type-term structure}.
\end{example}

Some natural transformations between interpretations can be
internalized. In fact, the domain need not be an interpretation.

\begin{definition}
  Let \(\tth\) be an \(\infty\)-type theory, \(\logos\) an
  \(\infty\)-logos, \(\itpr : \tth \to \logos\) an
  interpretation, and \(\fun : \tth \to \logos\) a functor preserving
  pullbacks of representable maps. A a natural transformation
  \(\trans : \fun \To \itpr\) is \emph{cartesian on representable
    maps} if for any representable map \(\mor : \objI \to \obj\) in
  \(\tth\), the canonical map
  \begin{math}
    \fun(\objI) \to \fun(\obj) \times_{\itpr(\obj)} \itpr(\objI)
  \end{math}
  is an equivalence.
\end{definition}

\begin{example}
  \label{exm:cartesian-natural-transformation-internal}
  Let \(\logos\) be an \(\infty\)-logos, \(\tth\) a cellular
  \(\infty\)-type theory, \(\itprI : \tth \to \logos\) an interpretation, and
  \(\itpr : \tth \to \logos\) a functor preserving pullbacks of
  representable maps. Although \(\itpr\) is not an interpretation, it
  has a component at \(\typeof : \Tm \to \Ty\) which induces a type-term
  structure \(\tytmstr\) in the internal language of \(\logos\). Let
  \(\tytmstrI\) be the type-term structure corresponding to
  \(\itprI\). A \emph{cartesian morphism \(\map\) from \(\tytmstr\) to
    \(\tytmstrI\)} consists of of a
  function
  \begin{math}
    \ilCorrApp_{\ilTy_{1}}(\map) :
    \tytmstr
    \to \tytmstrI
  \end{math}
  and an equivalence
  \begin{math}
    \ilCorrApp_{\ilTm_{1}}(\map) :
    \ilForall_{\ty : \tytmstr}
    \ty \simeq \ilCorrApp_{\ilTy_{1}}(\map, \ty)
  \end{math}
  that respect the operators of \(\tth\) in a suitable sense. Note
  that, since \(\ilCorrApp_{\ilTm_{1}}(\map)\) is an equivalence,
  \(\map\) also acts on families of \dots of families of
  types/terms. Any cartesian morphism from \(\tytmstr\) to
  \(\tytmstrI\) uniquely extends to a natural transformation cartesian
  on representable maps from \(\itpr\) to \(\itprI\).
\end{example}

Finally, an interpretation in an \(\infty\)-logos can be turned into a
model.

\begin{construction}
  \label{cst:extern-itpr}
  Let \(\tth\) be an \(\infty\)-type theory, \(\logos\) an
  \(\infty\)-logos, and \(\itpr : \tth \to \logos\) an
  interpretation. Suppose that a full subcategory
  \(\cat \subset \logos\) contains the terminal object and is closed under
  \(\itpr\)-context extension, that is,
  \(\ty^{\pbMark} \itpr(\objI) \in \cat\) for any representable map
  \(\objI \to \obj\) in \(\tth\), any object \(\ctx \in \cat\), and any
  map \(\ty : \ctx \to \itpr(\obj)\). Then the composite
  \begin{math}
    \tth
    \xrightarrow{\itpr}
    \logos
    \xrightarrow{\nerve_{\cat}}
    \RFib_{\cat}
  \end{math}
  is a morphism of \((\infty, 1)\)-categories with representable maps,
  where the representable maps in \(\RFib_{\cat}\) is the standard
  ones, and \(\nerve_{\cat}\) takes \(\sh \in \logos\) to the right
  fibration \(\cat \times_{\logos} \logos_{\slice \sh}\) over
  \(\cat\). This defines a model \(\extern{\itpr}_{\cat}\) of \(\tth\)
  such that \(\Ctx(\extern{\itpr}_{\cat}) \simeq \cat\) which we call the
  \emph{externalization of \(\itpr\) at \(\cat\)}.
\end{construction}

\section{Internalizing diagrams}
\label{sec:mode-sketches}

We review the technique of internalizing a diagram of
\(\infty\)-logoses to type theory \parencite{uemura2022diagrams-arxiv}. We
follow \parencite{hottbook} for notations in homotopy type theory. We
use curly brackets like
\begin{math}
  \map :
  \ilForall_{\implicit{\var : \ty}}
  \tyI(\var)
\end{math}
to indicate that the argument \(\var\) of \(\map\) is implicit.

\subsection{Modalities in homotopy type theory}
\label{sec:modal-homot-type}

We recall \emph{modalities in homotopy type theory}
\parencite{rijke2020modalities}. In this section, we work in type
theory. A \emph{lex, accessible modality} \(\mode\), {\acrLAM} for
short, consists of a function \(\ilIn_{\mode} : \ilUniv \to \ilProp\)
satisfying that the inclusion \(\ilUniv_{\mode} \to \ilUniv\), where
\(\ilUniv_{\mode} \defeq \{\ty : \ilUniv \mid \ilIn_{\mode}(\ty)\}\), has
a left adjoint \(\opMode_{\mode} : \ilUniv \to \ilUniv_{\mode}\) with
unit
\(\unitMode_{\mode} : \ilForall_{\ty : \ilUniv} \ty \to \opMode_{\mode}
\ty\), that \(\ilIn_{\mode}\) is closed under dependent pair types,
that \(\opMode_{\mode}\) preserves pullbacks, and a certain
``accessibility'' condition. A type satisfying \(\ilIn_{\mode}\) is
called \emph{\(\mode\)-modal}. More generally, we say a (many-sorted)
structure is \emph{\(\mode\)-modal} if its carrier types are
\(\mode\)-modal. Let \(\ilLAM\) denote the type of {\acrLAMs}. For
{\acrLAMs} \(\mode\) and \(\modeI\), we write
\(\mode \le {}^{\orthMark}\modeI\) if for any \(\mode\)-modal type
\(\ty\) and for any \(\modeI\)-modal type \(\tyI\), the diagonal
function
\begin{math}
  \tyI \to (\ty \to \tyI)
\end{math}
is an equivalence.

\begin{proposition}[{Fracture and gluing theorem}]
  \label{fact:fracture-and-gluing}
  Let \(\mode\) and \(\modeI\) be {\acrLAMs} such that
  \(\mode \le {}^{\orthMark}\modeI\).
  \begin{enumerate}
  \item A {\acrLAM} \(\mode \lor \modeI\) called the \emph{canonical
      join} exists.
  \item
    \begin{math}
      \ilUniv_{\mode \lor \modeI}
      \simeq \ilExists_{\tyI : \ilUniv_{\modeI}}\tyI \to \ilUniv_{\mode}
    \end{math}
    where the right-to-left function sends a \((\tyI, \ty)\) to
    \(\ilExists_{\var : \tyI}\ty(\var)\).
  \end{enumerate}
\end{proposition}
\begin{proof}
  All but the accessibility of \(\mode \lor \modeI\) follows from
  \parencite[Theorem 3.50]{rijke2020modalities}. See
  \parencite[Proposition 2.18]{uemura2022diagrams-arxiv} for the
  accessibility.
\end{proof}

\subsection{Mode sketches}
\label{sec:mode-sketches-1}

We work in type theory. For {\acrLAMs} \(\mode_{0}\), \(\mode_{1}\),
\(\mode_{2}\), we define a function
\(\opMode^{\mode_{1}}_{\mode_{0}} : \ilUniv_{\mode_{1}} \to
\ilUniv_{\mode_{0}}\) to be the restriction of \(\opMode_{\mode_{0}}\)
to \(\ilUniv_{\mode_{1}} \subset \ilUniv\) and a function
\begin{math}
  \unitMode^{\mode_{0}; \mode_{2}}_{\mode_{1}} :
  \ilForall_{\ty : \ilUniv_{\mode_{2}}}
  \opMode^{\mode_{2}}_{\mode_{0}} \ty \to
  \opMode^{\mode_{1}}_{\mode_{0}} \opMode^{\mode_{2}}_{\mode_{1}} \ty
\end{math}
by
\begin{math}
  \unitMode^{\mode_{0}; \mode_{2}}_{\mode_{1}}(\ty) \defeq
  \opMode_{\mode_{0}} \unitMode_{\mode_{1}}(\ty).
\end{math}

\begin{definition}
  A \emph{mode sketch} \(\modesketch\) consists of a decidable finite
  poset \(\idxModesketch_{\modesketch}\) and a subset
  \(\triModesketch_{\modesketch}\) of triangles in
  \(\idxModesketch_{\modesketch}\) called \emph{thin} triangles. Here,
  by a \emph{decidable} poset we mean a poset whose ordering relation
  is decidable, and by a \emph{triangle} in
  \(\idxModesketch_{\modesketch}\) we mean a strictly ordered triple
  \((\idx_{0} < \idx_{1} < \idx_{2})\) of elements of
  \(\idxModesketch_{\modesketch}\). Note that this definition also
  makes sense in the metatheory. Every mode sketch in the metatheory
  can be encoded in type theory since it is finite.
\end{definition}

Let \(\modesketch\) be a mode sketch and let \(\mode : \modesketch \to
\ilLAM\). We consider the following axioms.

\begin{axiom}
  \label{axm:mode-sketch-orthogonal}
  \begin{math}
    \mode(\idx) \le {}^{\orthMark}\mode(\idxI)
  \end{math}
  for any \(\idxI \not\le \idx\) in \(\modesketch\).
\end{axiom}

\begin{axiom}
  \label{axm:mode-sketch-thin}
  For any thin triangle \((\idx_{0} < \idx_{1} < \idx_{2})\) in
  \(\modesketch\), the function
  \begin{math}
    \unitMode^{\mode(\idx_{0});
      \mode(\idx_{2})}_{\mode(\idx_{1})}(\ty) :
    \opMode^{\mode(\idx_{2})}_{\mode(\idx_{0})} \ty \to
    \opMode^{\mode(\idx_{1})}_{\mode(\idx_{0})}
    \opMode^{\mode(\idx_{2})}_{\mode(\idx_{1})} \ty
  \end{math}
  is an equivalence for every \(\ty : \ilUniv_{\mode(\idx_{2})}\).
\end{axiom}

\begin{remark}
  Let \(\modesketch\) be a mode sketch and let
  \(\mode : \modesketch \to \ilLAM\) be a function satisfying
  \cref{axm:mode-sketch-orthogonal}. For any decidable subset
  \(\idxty \subset \modesketch\), the canonical join
  \(\biglor_{\idx : \idxty} \mode(\idx)\) exists by
  \cref{fact:fracture-and-gluing}.
\end{remark}

\subsection{Models of mode sketches}
\label{sec:models-mode-sketches}

We work in the metatheory.

\begin{construction}
  Let \(\modesketch\) be a mode sketch. It presents an
  \((\infty, 2)\)-category \(\realize{\modesketch}\) defined as
  follows. The \(0\)-cells are the elements of \(\modesketch\). The
  \(1\)-cells are generated by \((\idx < \idxI) : \idx \to \idxI\) for
  every \(\idx < \idxI\) in \(\modesketch\). The \(2\)-cells are
  generated by
  \((\idx_{0} < \idx_{1} < \idx_{2}) : (\idx_{1} < \idx_{2}) \comp
  (\idx_{0} < \idx_{1}) \To (\idx_{0} < \idx_{2})\) for every triangle
  \(\idx_{0} < \idx_{1} < \idx_{2}\) in \(\modesketch\). A generating
  \(2\)-cell is made invertible when it is thin. For every chain
  \(\idx_{0} < \idx_{1} < \dots < \idx_{\nat}\) in \(\modesketch\), a
  certain diagram in \(\realize{\modesketch}\) is filled by a
  homotopy.
\end{construction}

\begin{theorem}
  \label{thm:mode-sketch-model}
  Let \(\modesketch\) be a mode sketch and let
  \(\cat : \realize{\modesketch} \to \Cat^{\nMark{2}}\) be a
  functor. Then \(\Fun(\realize{\modesketch}, \RFib)_{\cat}\) is an
  \(\infty\)-logos equipped with a function
  \(\mode : \modesketch \to \ilLAM\) in its internal language satisfying
  \cref{axm:mode-sketch-orthogonal,axm:mode-sketch-thin}. Moreover,
  for any \(\idx \in \modesketch\), the function
  \(\opMode_{\mode(\idx)} : \ilUniv \to \ilUniv_{\mode(\idx)}\) is
  interpreted as the restriction functor
  \(\Fun(\realize{\modesketch}, \RFib)_{\cat} \to \RFib_{\cat(\idx)}\).
\end{theorem}
\begin{proof}
  This follows from \parencite[Theorem 10.3]{uemura2022diagrams-arxiv}
  because \(\Fun(\realize{\modesketch}, \RFib)_{\cat}\) is the oplax
  limit of the diagram \(\idx \mapsto \RFib_{\cat(\idx)}\).
\end{proof}

\begin{remark}
  \label{rem:modal-tininess}
  Let \(\modesketch\) be a mode sketch and \(\tth\) a cellular
  \(\infty\)-type theory. By the definition of morphisms of models of
  \(\tth\), a diagram of models of \(\tth\) indexed over
  \(\realize{\modesketch}\) is a pair \((\cat, \fun)\) consisting of a
  functor \(\cat : \realize{\modesketch} \to \Cat^{\nMark{2}}\) and a
  morphism of \((\infty, 1)\)-categories with representable maps
  \(\fun : \tth \to \Fun(\realize{\modesketch}, \RFib)_{\cat}\). Let
  \(\tytmstr\) be the \(\tth\)-type-term structure in the internal
  language of \(\Fun(\realize{\modesketch}, \RFib)_{\cat}\)
  corresponding to \(\fun\). For any \(\idx \in \modesketch\), since the
  restriction functor
  \(\Fun(\realize{\modesketch}, \RFib)_{\cat} \to \RFib_{\cat_{\idx}}\)
  is a morphism of \(\infty\)-categories of representable maps, for any
  \(\ty : \tytmstr\) and for any
  \(\tyX : \ty \to \ilUniv_{\biglor_{\idxI :
      \modesketch}\mode(\idxI)}\), the canonical map
  \begin{equation*}
    \opMode_{\mode(\idx)}(\ilForall_{\tm : \ty}\tyX(\tm))
    \to (\ilForall_{\tm : \ty}\opMode_{\mode(\idx)} \tyX(\tm))
  \end{equation*}
  is an equivalence. This condition is equivalent to the
  \emph{preservation of virtual context extensions} in the terminology
  of \textcite{bocquet2021relative}.
\end{remark}

\subsection{Logical relations as types}
\label{sec:logical-relations-as}

We work in type theory. Let \(\mode\) and \(\modeI\) be {\acrLAMs}
such that \(\mode \le
{}^{\orthMark}\modeI\). \Cref{fact:fracture-and-gluing} asserts that a
type in \(\ilUniv_{\mode \lor \modeI}\) is fractured into a
\(\modeI\)-modal type equipped with a \(\mode\)-modal unary
(proof-relevant) relation on it. \(\ilUniv_{\mode \lor \modeI}\)
allows us to work with logical relations synthetically, summarized by
the slogan ``logical relations as types'' of
\textcite{sterling2021logical}.

One can show that the equivalence in
\cref{fact:fracture-and-gluing} respects arbitrary higher-order
structure in the following sense. Let \(\metaty\) be a construction of
a type only using universes, the unit type, dependent pair types,
dependent function types, and identity types. For any sufficiently
large universe \(\ilUniv\) and for any {\acrLAM} \(\modeI\), we can
translate \(\metaty\) into a type
\(\trsl{\metaty}_{\ilUniv_{\modeI}} : \ilUniv_{\modeI}\). Note that
\(\trsl{\blank}\) is a syntactic translation which is not internal to
type theory. For another {\acrLAM} \(\mode\), the \emph{logical
  relation translation}
\parencite{bernardy2012proofs,shulman2015inverse,uemura2017fibred,lasson2014canonicity}
\(\trsl{\metaty}^{\neord[1]}_{\ilUniv_{\modeI}, \ilUniv_{\mode}} :
\trsl{\metaty}_{\ilUniv_{\modeI}} \to \ilUniv_{\mode}\) is defined.

\begin{theorem}
  \label{thm:higher-order-fracture}
  Under the above notation, if \(\mode \le {}^{\orthMark}\modeI\), then
  \(\trsl{\metaty}_{\ilUniv_{\mode \lor \modeI}} : \ilUniv_{\mode \lor
    \modeI}\) corresponds via the equivalence in
  \cref{fact:fracture-and-gluing} to the pair
  \begin{math}
    (\trsl{\metaty}_{\ilUniv_{\modeI}},
    \trsl{\metaty}^{\neord[1]}_{\ilUniv_{\modeI}, \ilUniv_{\mode}}),
  \end{math}
  or in other words we have
  \begin{equation}
    \label{eq:fracture-structure-1}
    \trsl{\metaty}_{\ilUniv_{\mode \lor \modeI}} \simeq
    \ilExists_{\var : \trsl{\metaty}_{\ilUniv_{\modeI}}}\trsl{\metaty}^{\neord[1]}_{\ilUniv_{\modeI}, \ilUniv_{\mode}}(\var).
  \end{equation}
\end{theorem}
\begin{proof}
  By case analysis on \(\metaty\); see \parencite[Section
  4.2]{uemura2022diagrams-arxiv} for how the equivalence in
  \cref{fact:fracture-and-gluing} interacts with each type
  constructor.
\end{proof}

When \(\metaty\) is the construction of the type of a certain
(possibly higher-order) structure, then
\(\trsl{\metaty}^{\neord[1]}_{\ilUniv_{\modeI},
  \ilUniv_{\mode}}(\var)\) is considered as the type of
\emph{displayed} structures over
\(\var : \trsl{\metaty}_{\ilUniv_{\modeI}}\)
\parencite[cf.][]{ahrens2017displayed}. Thus,
\cref{eq:fracture-structure-1} asserts that any
\((\mode \lor \modeI)\)-modal structure is fractured into a
\(\modeI\)-modal structure and a \(\mode\)-modal displayed structure
over it.

The logical relation translation is parameterized by \emph{inverse
  categories}, categories in which there is no infinite ascending
chain of non-identity morphisms. For every inverse category
\(\idxsh\), we can define a syntactic translation
\(\trsl{\blank}^{\idxsh}\) which is a syntactic counterpart of the
semantics in inverse diagrams \parencite{shulman2015inverse}. Any mode
sketch \(\modesketch\) is in particular an inverse category, and thus
we have a theorem analogous to \cref{thm:higher-order-fracture}: for
any \(\mode : \modesketch \to \ilLAM\) satisfying
\cref{axm:mode-sketch-orthogonal},
\(\trsl{\metaty}_{\ilUniv_{\biglor_{\idx : \modesketch}\mode(\idx)}}\)
is fractured in a certain way.

\section{Normal forms}
\label{sec:normal-forms-1}

From now on, let us fix a cellular \(\infty\)-type theory \(\workingtth\), and let
\(\workingInitModel\) be the initial model of \(\workingtth\). We explain a
\emph{strategy} to solve \cref{prb:coherence-problem} using a
\emph{normalization} technique. The strategy may or may not work
depending on (the presentation of) the \(\infty\)-type theory
\(\workingtth\). Heuristically, it works when all the equations in the
presentation of the \(\infty\)-type theory have clear computational
meaning.

Each of
\cref{sec:normal-forms-1,sec:normalization-models,sec:norm-results,sec:coherence-theorem}
is organized as follows. We first construct a working
\(\infty\)-logos. We then list relevant axioms satisfied by the working
\(\infty\)-logos. We work in type theory to define some concepts, construct
some objects, and prove some propositions, assuming some of the axioms
when needed. Optionally, we externalize the results to obtain results
in the metatheory. Paragraphs marked `Task' explain what we have to do
for each individual \(\infty\)-type theory. In the last subsection, we
demonstrate how the tasks are performed for a simple example of
\(\infty\)-type theory.

The first step is to construct spaces of \emph{normal forms}.

\subsection{Semantic setup}
\label{sec:semantic-setup}

Normal forms are \emph{not} stable under arbitrary substitution, but
they are stable under \emph{renaming of variables}. Therefore, they
should live in the \(\infty\)-logos of right fibrations over the
\((\infty, 1)\)-category of contexts and renamings rather than all
substitutions. The following definition in the \(1\)-categorical case
is due to \textcite{bocquet2021relative}.

\begin{definition}
  Let \(\model\) be a model of \(\basicTT\). A \emph{renaming model
    over \(\model\)} is a model \(\modelI\) of \(\basicTT\) equipped
  with a morphism \(\fun : \modelI \to \model\) such that the map
  \(\Ty(\modelI) \to \fun_{\Ctx}^{\pbMark} \Ty(\model)\) is an
  equivalence.
\end{definition}

This definition is essentially algebraic. In particular, the initial
renaming model \(\Ren_{\model}\) over
\(\model\) exists. By definition, \(\Ren_{\model}\) has the same types
as \(\model\), but the initiality of \(\Ren_{\model}\) means that it
contains only trivial terms. Since \(\basicTT\) does not have any type
or term constructor, the only trivial terms are variables.

For the definition of normal forms and the construction of the
normalization model in \cref{sec:normalization-models}, the initiality
of \(\workingInitModel\) or \(\Ren_{\model}\) is not
necessary. Although we do not need much generality, let us work with
an arbitrary renaming model \(\fun : \modelI \to \model\) over an
arbitrary model \(\model\) of \(\workingtth\) to simplify arguments by
forgetting about special properties of the initial models. Our first
working \(\infty\)-logos \(\workingLogos_{\fun}\) is
\(\Fun(\neord[1], \RFib)_{\fun_{\Ctx}}\).

\subsection{Axiomatic setup}
\label{sec:axiomatic-setup}

The \(\infty\)-logos \(\workingLogos_{\fun}\) satisfies several axioms. Let
\(\workingMS\) be the mode sketch
\begin{math}
  \{0 \to 1\}.
\end{math}
Note that there is no thin triangle as there is no triangle. By
\cref{thm:mode-sketch-model}, \(\workingLogos_{\fun}\) satisfies the
associated mode sketch axioms.

\begin{axiom}
  \label{axm:working-logos-1-mode-sketch}
  A function \(\mode : \workingMS \to \ilLAM\) satisfying
  \cref{axm:mode-sketch-orthogonal,axm:mode-sketch-thin} is
  postulated. We define \(\modeRen \defeq \mode(0)\) and
  \(\modeSyn \defeq \mode(1)\).
\end{axiom}

The morphism \(\fun : \workingtth \to \workingLogos_{\fun}\) corresponds to a
\((\modeRen \lor \modeSyn)\)-modal type-term structure in the internal
language of \(\workingLogos_{\fun}\). Its \(\ilTy\)-component is
\(\modeSyn\)-modal since \(\fun\) is a renaming model. It is then
fractured into the following structures
(\cref{axm:working-logos-1-type-term-structure,axm:working-logos-1-variable}).

\begin{axiom}
  \label{axm:working-logos-1-type-term-structure}
  Presuppose \cref{axm:working-logos-1-mode-sketch}. A
  \(\modeSyn\)-modal \(\workingtth\)-type-term structure \(\workingTyTmStr\)
  is postulated.
\end{axiom}

\begin{definition}
  Let \(\tytmstr\) be a type-term structure. A \emph{variable
    structure on \(\tytmstr\)} is a type family
  \begin{math}
    \ilIsVar :
    \ilForall_{\implicit{\ty : \tytmstr}}
    \ty
    \to \ilUniv.
  \end{math}
\end{definition}

\begin{axiom}
  \label{axm:working-logos-1-variable}
  Presuppose
  \cref{axm:working-logos-1-mode-sketch,axm:working-logos-1-type-term-structure}. A
  \(\modeRen\)-modal variable structure \(\ilIsVar\) on
  \(\workingTyTmStr\) is postulated.
\end{axiom}

\subsection{Normal forms}
\label{sec:normal-forms}

We work in type theory.

\begin{definition}
  \label{def:normal-form-predicate}
  Let \(\tytmstr\) be a type-term structure. A \emph{normal form
    predicate \(\nfalg\) on \(\tytmstr\)} consists of the following
  data.
  \begin{align*}
    \begin{autobreak}
      \MoveEqLeft
      \ilIsNfTy_{\nfalg} :
      \tytmstr
      \to \ilUniv
    \end{autobreak}
    \\
    \begin{autobreak}
      \MoveEqLeft
      \ilIsNfTm_{\nfalg} :
      \ilForall_{\implicit{\ty : \tytmstr}}
      \ty
      \to \ilUniv
    \end{autobreak}
    \\
    \begin{autobreak}
      \MoveEqLeft
      \ilIsNeTm_{\nfalg} :
      \ilForall_{\implicit{\ty : \tytmstr}}
      \ty
      \to \ilUniv
    \end{autobreak}
  \end{align*}
\end{definition}

\begin{definition}
  \label{def:normal-form-predicate-variable}
  Let \(\tytmstr\) be a type-term structure equipped with a variable
  structure and let \(\nfalg\) be a normal form predicate on
  \(\tytmstr\). We say \(\nfalg\) \emph{supports variables} if it is
  equipped with the following operator.
  \begin{align*}
    \begin{autobreak}
      \MoveEqLeft
      \ilVarToNe :
      \ilForall_{\implicit{\ty : \tytmstr}}
      \ilForall_{\implicit{\var : \ty}}
      \ilIsVar(\var)
      \to \ilIsNeTm_{\nfalg}(\var)
    \end{autobreak}
  \end{align*}
\end{definition}

The definitions of normal forms and neutral terms depend on the
presentation of the cellular \(\infty\)-type theory \(\workingtth\).

\begin{task}
  \label{task:normal-form-algebra}
  Let \(\tytmstr\) be a \(\workingtth\)-type-term structure equipped with a
  variable structure. Extend \cref{def:normal-form-predicate} to a
  suitable notion of a \emph{\(\workingtth\)-normal form algebra}. One must
  define it as a signature for mutual inductive type families so that
  the initial \(\workingtth\)-normal form algebra supporting variables
  exists.
\end{task}

\begin{construction}
  Assume
  \cref{axm:working-logos-1-mode-sketch,axm:working-logos-1-type-term-structure,axm:working-logos-1-variable}. Combining
  the signature for \(\workingtth\)-normal form algebra and the nullification
  technique of \textcite[Section 2.2]{rijke2020modalities}, one can
  construct \(\workingNfalg\) the initial \(\modeRen\)-modal
  \(\workingtth\)-normal form algebra supporting variables.
\end{construction}

\subsection{Running example}
\label{sec:running-example}

We use an \(\infty\)-type theory \(\basicTT_{\FunMark}\) with function
types and a natural numbers type as a running example to demonstrate
how to perform the tasks in the normalization
strategy. \(\basicTT_{\FunMark}\) also includes constants
\(\ilConstTy\), \(\ilConstTyI\), and \(\ilConstTm\) by which we
explain how normalization works not only for the initial model but
also free models.

The notion of a \emph{\(\basicTT_{\FunMark}\)-type-term structure}
extends the notion of a type-term structure \(\tytmstr\) by the
following operators.
\begin{align*}
  \begin{autobreak}
    \MoveEqLeft
    \ilFun :
    \ilForall_{\ty : \tytmstr}
    (\ty \to \tytmstr)
    \to \tytmstr
  \end{autobreak}
  \\
  \begin{autobreak}
    \MoveEqLeft
    \ilabs :
    \ilForall_{\implicit{\ty : \tytmstr}}
    \ilForall_{\implicit{\tyI : \ty \to \tytmstr}}
    (\ilForall_{\tm : \ty}\tyI(\tm)) \simeq \ilFun(\ty, \tyI)
  \end{autobreak}
  \\
  \begin{autobreak}
    \MoveEqLeft
    \ilNat :
    \tytmstr
  \end{autobreak}
  \\
  \begin{autobreak}
    \MoveEqLeft
    \ilzero :
    \ilNat
  \end{autobreak}
  \\
  \begin{autobreak}
    \MoveEqLeft
    \ilsucc :
    \ilNat
    \to \ilNat
  \end{autobreak}
  \\
  \begin{autobreak}
    \MoveEqLeft
    \ilind_{\ilNat} :
    \ilForall_{\nat : \ilNat}
    \ilForall_{\ty : \ilNat \to \tytmstr}
    \ty(\ilzero)
    \to (\ilForall_{\var}\ty(\var) \to \ty(\ilsucc(\var)))
    \to \ty(\nat)
  \end{autobreak}
  \\
  \begin{autobreak}
    \MoveEqLeft
    \ilbeta_{\ilzero} :
    \ilForall_{\ty}
    \ilForall_{\tm_{\zeroMark}}
    \ilForall_{\tm_{\succMark}}
    \ilind_{\ilNat}(\ilzero, \ty, \tm_{\zeroMark}, \tm_{\succMark}) = \tm_{\zeroMark}
  \end{autobreak}
  \\
  \begin{autobreak}
    \MoveEqLeft
    \ilbeta_{\ilsucc} :
    \ilForall_{\nat}
    \ilForall_{\ty}
    \ilForall_{\tm_{\zeroMark}}
    \ilForall_{\tm_{\succMark}}
    \ilind_{\ilNat}(\ilsucc(\nat), \ty, \tm_{\zeroMark}, \tm_{\succMark}) = \tm_{\succMark}(\blank, \ilind_{\ilNat}(\nat, \ty, \tm_{\zeroMark}, \tm_{\succMark}))
  \end{autobreak}
  \\
  \begin{autobreak}
    \MoveEqLeft
    \ilConstTy :
    \tytmstr
  \end{autobreak}
  \\
  \begin{autobreak}
    \MoveEqLeft
    \ilConstTyI :
    \ilConstTy
    \to \tytmstr
  \end{autobreak}
  \\
  \begin{autobreak}
    \MoveEqLeft
    \ilConstTm :
    \ilForall_{\tm : \ilConstTy}
    \ilConstTyI(\tm)
  \end{autobreak}
\end{align*}
We define \(\ilapp(\map, \tm) \defeq \ilabs^{-1}(\map)(\tm)\) for
\(\map : \ilFun(\ty, \tyI)\) and \(\tm : \ty\).

\paragraph{\Cref{task:normal-form-algebra}}

The notion of a \emph{\(\basicTT_{\FunMark}\)-normal form algebra}
extends the notion of a normal form predicate \(\nfalg\) on
\(\tytmstr\) by the following operators.
\begin{align*}
  \begin{autobreak}
    \MoveEqLeft
    \ilFun_{\nftyMark} :
    \ilForall_{\implicit{\ty : \tytmstr}}
    \ilIsNfTy_{\nfalg}(\ty)
    \to \ilForall_{\implicit{\tyI : \ty \to \tytmstr}}
    (\ilForall_{\var : \ty}\ilIsVar(\var) \to \ilIsNfTy_{\nfalg}(\tyI(\var)))
    \to \ilIsNfTy_{\nfalg}(\ilFun(\ty, \tyI))
  \end{autobreak}
  \\
  \begin{autobreak}
    \MoveEqLeft
    \ilabs_{\nftmMark} :
    \ilForall_{\implicit{\ty : \tytmstr}}
    \ilForall_{\implicit{\tyI : \ty \to \tytmstr}}
    \ilForall_{\implicit{\tmI : \ilForall_{\tm}\tyI(\tm)}}
    (\ilForall_{\var : \ty}\ilIsVar(\var) \to \ilIsNfTm_{\nfalg}(\tmI(\var)))
    \to \ilIsNfTm_{\nfalg}(\ilabs(\tmI))
  \end{autobreak}
  \\
  \begin{autobreak}
    \MoveEqLeft
    \ilapp_{\netmMark} :
    \ilForall_{\implicit{\ty : \tytmstr}}
    \ilForall_{\implicit{\tyI : \ty \to \tytmstr}}
    \ilForall_{\implicit{\map : \ilFun(\ty, \tyI)}}
    \ilIsNeTm_{\nfalg}(\map)
    \to \ilForall_{\implicit{\tm : \ty}}
    \ilIsNfTm_{\nfalg}(\tm)
    \to \ilIsNeTm_{\nfalg}(\ilapp(\map, \tm))
  \end{autobreak}
  \\
  \begin{autobreak}
    \MoveEqLeft
    \ilNat_{\nftyMark} :
    \ilIsNfTy_{\nfalg}(\ilNat)
  \end{autobreak}
  \\
  \begin{autobreak}
    \MoveEqLeft
    \ilzero_{\nftmMark} :
    \ilIsNfTm_{\nfalg}(\ilzero)
  \end{autobreak}
  \\
  \begin{autobreak}
    \MoveEqLeft
    \ilsucc_{\nftmMark} :
    \ilForall_{\implicit{\nat}}
    \ilIsNfTm_{\nfalg}(\nat)
    \to \ilIsNfTm_{\nfalg}(\ilsucc(\nat))
  \end{autobreak}
  \\
  \begin{autobreak}
    \MoveEqLeft
    \ilNeToNf_{\ilNat} :
    \ilForall_{\implicit{\nat : \ilNat}}
    \ilIsNeTm_{\nfalg}(\nat)
    \to \ilIsNfTm_{\nfalg}(\nat)
  \end{autobreak}
  \\
  \begin{autobreak}
    \MoveEqLeft
    \ilind_{\ilNat, \netmMark} :
    \ilForall_{\implicit{\nat}}
    \ilIsNeTm_{\nfalg}(\nat)
    \to \ilForall_{\implicit{\ty}}
    (\ilForall_{\var}\ilIsVar(\var) \to \ilIsNfTy_{\nfalg}(\ty(\var)))
    \to \ilForall_{\implicit{\tm_{\zeroMark}}}
    \ilIsNfTm_{\nfalg}(\tm_{\zeroMark})
    \to \ilForall_{\implicit{\tm_{\succMark}}}
    (\ilForall_{\var}\ilIsVar(\var) \to \ilForall_{\varI}\ilIsVar(\varI) \to \ilIsNfTm_{\nfalg}(\tm_{\succMark}(\var, \varI)))
    \to \ilIsNeTm_{\nfalg}(\ilind_{\ilNat}(\nat, \ty, \tm_{\zeroMark}, \tm_{\succMark}))
  \end{autobreak}
  \\
  \begin{autobreak}
    \MoveEqLeft
    \ilConstTy_{\nftyMark} :
    \ilIsNfTy_{\nfalg}(\ilConstTy)
  \end{autobreak}
  \\
  \begin{autobreak}
    \MoveEqLeft
    \ilNeToNf_{\ilConstTy} :
    \ilForall_{\implicit{\tmII : \ilConstTy}}
    \ilIsNeTm_{\nfalg}(\tmII)
    \to \ilIsNfTm_{\nfalg}(\tmII)
  \end{autobreak}
  \\
  \begin{autobreak}
    \MoveEqLeft
    \ilConstTyI_{\nftyMark} :
    \ilForall_{\implicit{\tmII}}
    \ilIsNfTm_{\nfalg}(\tmII)
    \to \ilIsNfTy_{\nfalg}(\ilConstTyI(\tmII))
  \end{autobreak}
  \\
  \begin{autobreak}
    \MoveEqLeft
    \ilNeToNf_{\ilConstTyI} :
    \ilForall_{\implicit{\tmII}}
    \ilIsNfTm_{\nfalg}(\tmII)
    \to \ilForall_{\implicit{\tmIII : \ilConstTyI(\tmII)}}
    \ilIsNeTm_{\nfalg}(\tmIII)
    \to \ilIsNfTm_{\nfalg}(\tmIII)
  \end{autobreak}
  \\
  \begin{autobreak}
    \MoveEqLeft
    \ilConstTm_{\netmMark} :
    \ilForall_{\implicit{\tmII}}
    \ilIsNfTm_{\nfalg}(\tmII)
    \to \ilIsNeTm_{\nfalg}(\ilConstTm(\tmII))
  \end{autobreak}
\end{align*}
The definition of normal forms and neutral terms for an inductive type
fits the following pattern: the type is in normal form if its
arguments are all in normal form (\(\ilNat_{\nftyMark}\)); a
constructor is in normal form if its arguments are all in normal form
(\(\ilzero_{\nftyMark}\) and \(\ilsucc_{\nftmMark}\)); any neutral
term of the type is in normal form (\(\ilNeToNf_{\ilNat}\)); the
eliminator is neutral if its head argument is neutral and the rest of
the arguments are all in normal form (\(\ilind_{\ilNat,
  \netmMark}\)). The pattern for normal forms and neutral terms for
constants is as follows: a type constant is in normal form if its
arguments are all in normal form (\(\ilConstTy_{\nftyMark}\) and
\(\ilConstTyI_{\nftyMark}\)); any neutral term of a type constant is
in normal form (\(\ilNeToNf_{\ilConstTy}\) and
\(\ilNeToNf_{\ilConstTyI}\)); a term constant is neutral if its
arguments are all in normal form (\(\ilConstTm_{\netmMark}\)).

\begin{remark}
  There are several choices of arguments of constructors. For example,
  some authors \parencite[e.g.][]{coquand2019canonicity} choose
  \begin{align*}
    \begin{autobreak}
      \MoveEqLeft
      \ilabs_{\nftmMark}' :
      \ilForall_{\implicit{\ty : \tytmstr}}
      \ilIsNfTy_{\nfalg}(\ty)
      \to \ilForall_{\implicit{\tyI : \ty \to \tytmstr}}
      (\ilForall_{\var}\ilIsVar(\var) \to \ilIsNfTy_{\nfalg}(\tyI(\var)))
      \to \ilForall_{\implicit{\tmI : \ilForall_{\tm}\tyI(\tm)}}
      (\ilForall_{\var}\ilIsVar(\var) \to \ilIsNfTm_{\nfalg}(\tmI(\var)))
      \to \ilIsNfTm_{\nfalg}(\ilabs(\tmI))
    \end{autobreak}
  \end{align*}
  instead of our \(\ilabs_{\nftmMark}\). This difference is not
  significant: we will prove in \cref{sec:norm-results} the
  contractibility of \(\ilIsNfTy_{\nfalg}(\ty)\) and
  \(\ilIsNfTm_{\nfalg}(\tm)\), so all variants are equivalent as long
  as the proof works. In principle, we prefer fewer arguments for
  simplicity. An exception is the \(\ilNeToNf_{\ilConstTyI}\)
  constructor which takes a seemingly redundant argument
  \(\tmII_{\nftmMark} : \ilIsNfTm_{\nfalg}(\tmII)\). The reason for
  this choice is that in some proof by induction given in
  \cref{sec:coherence-theorem} we use the induction hypothesis for
  \(\tmII_{\nftmMark}\). We do not know whether this is really a
  necessary choice or there is a way to avoid using the induction
  hypothesis for \(\tmII_{\nftmMark}\).
\end{remark}

\section{Normalization models}
\label{sec:normalization-models}

We construct a \emph{logical relation} for normalization and then a
\emph{normalization model}. We continue working with the
\(\infty\)-logos \(\workingLogos_{\fun}\) introduced in
\cref{sec:normal-forms-1}.

\subsection{Normalization structures}
\label{sec:norm-struct}

We work in type theory. The notion of a \emph{displayed type-term
  structure} is defined by the logical relation translation of the
notion of a type-term structure with respect to the inverse category
\(\neord[1]\). For the record, let us unwind the definition.

\begin{definition}
  Let \(\tytmstr\) be a type-term structure. A \emph{displayed
    type-term structure \(\tytmstrI\) over \(\tytmstr\)} consists of
  the following data.
  \begin{align*}
    \begin{autobreak}
      \MoveEqLeft
      \ilTy^{\neord[1]}_{\tytmstrI} :
      \tytmstr
      \to \ilUniv
    \end{autobreak}
    \\
    \begin{autobreak}
      \MoveEqLeft
      \ilTm^{\neord[1]}_{\tytmstrI} :
      \ilForall_{\ty : \tytmstr}
      \ilTy^{\neord[1]}_{\tytmstrI}(\ty)
      \to \ty
      \to \ilUniv
    \end{autobreak}
  \end{align*}
  We regard \(\tytmstrI \mapsto \ilTy^{\neord[1]}_{\tytmstrI}\) as an
  implicit coercion from displayed type-term structures to
  \(\tytmstr \to \ilUniv\) and \(\ilTm^{\neord[1]}_{\tytmstrI}\) as an
  implicit coercion from \(\ilTy^{\neord[1]}_{\tytmstr}(\ty)\) to
  \(\ty \to \ilUniv\). When \(\tytmstr\) is a \(\workingtth\)-type-term
  structure, the notion of a \emph{displayed \(\workingtth\)-type-term
    structure over \(\tytmstr\)} is also defined.
\end{definition}

A displayed \(\workingtth\)-type-term structure is thought of as a
(proof-relevant) \emph{logical relation}. The \emph{fundamental lemma
  of logical relations} asserts that a model and a logical relation on
it are glued together to yield another model. We achieve this
internally as a special case of \cref{thm:higher-order-fracture}.

\begin{corollary}
  \label{cor:fracture-and-gluing-type-term-structure}
  Let \(\mode\) and \(\modeI\) be {\acrLAMs} such that
  \(\mode \le {}^{\orthMark} \modeI\). Then the type of
  \((\mode \lor \modeI)\)-modal \(\workingtth\)-type-term structures is
  equivalent to the type of pairs \((\tytmstr, \tytmstrI)\) consisting
  of a \(\modeI\)-modal \(\workingtth\)-type-term structure \(\tytmstr\) and
  a \(\mode\)-modal displayed \(\workingtth\)-type-term structure
  \(\tytmstrI\) over \(\tytmstr\) \qed
\end{corollary}

We consider a logical relation that interacts with normal forms.

\begin{definition}
  \label{def:normalization-structure}
  Let \(\tytmstrI\) be a displayed type-term structure over a
  type-term structure \(\tytmstr\), and let \(\nfalg\) be a normal
  form predicate on \(\tytmstr\). A \emph{\(\nfalg\)-normalization
    structure on \(\tytmstrI\)} consists of the following data.
  \begin{align*}
    \begin{autobreak}
      \MoveEqLeft
      \ilnfty :
      \ilForall_{\implicit{\ty : \tytmstr}}
      \tytmstrI(\ty)
      \to \ilIsNfTy_{\nfalg}(\ty)
    \end{autobreak}
    \\
    \begin{autobreak}
      \MoveEqLeft
      \ilreify :
      \ilForall_{\implicit{\ty : \tytmstr}}
      \ilForall_{\tyI : \tytmstrI(\ty)}
      \ilForall_{\implicit{\tm : \ty}}
      \tyI(\tm)
      \to \ilIsNfTm_{\nfalg}(\tm)
    \end{autobreak}
    \\
    \begin{autobreak}
      \MoveEqLeft
      \ilreflect :
      \ilForall_{\implicit{\ty : \tytmstr}}
      \ilForall_{\tyI : \tytmstrI(\ty)}
      \ilForall_{\implicit{\tm : \ty}}
      \ilIsNeTm_{\nfalg}(\tm)
      \to \tyI(\tm)
    \end{autobreak}
  \end{align*}
\end{definition}

\begin{construction}
  \label{cst:var-to-term}
  Let \(\tytmstr\) be a type-term structure equipped with a variable
  structure, \(\nfalg\) a normal form predicate on \(\tytmstr\)
  supporting variables, and \(\tyI\) a displayed type-term structure
  over \(\tytmstr\) equipped with a \(\nfalg\)-normalization
  structure. We define
  \begin{align*}
    \begin{autobreak}
      \MoveEqLeft
      \ilVarToTm :
      \ilForall_{\implicit{\ty : \tytmstr}}
      \ilForall_{\tyI : \tytmstrI(\ty)}
      \ilForall_{\implicit{\var : \ty}}
      \ilIsVar(\var)
      \to \tyI(\var)
    \end{autobreak}
    \\
    \begin{autobreak}
      \MoveEqLeft
      \ilVarToTm(\tyI, \var_{\varMark}) \defeq
      \ilreflect(\tyI, \ilVarToNe(\var_{\varMark})).
    \end{autobreak}
  \end{align*}
\end{construction}

A normalization structure should follow certain computation rules,
depending on the presentation of \(\workingtth\).

\begin{task}
  \label{task:normalization-structure}
  Let \(\tytmstr\) be a \(\workingtth\)-type-term structure equipped with a
  variable structure, \(\nfalg\) a \(\workingtth\)-normal form algebra on
  \(\tytmstr\) supporting variables, and \(\tytmstrI\) a displayed
  \(\workingtth\)-type-term structure over \(\tytmstr\). Extend
  \cref{def:normalization-structure} to a suitable notion of a
  \(\nfalg\)-normalization structure on \(\tytmstrI\) \emph{supporting
    \(\workingtth\)}. This is usually done by specifying how
  \(\ilnfty(\tyI)\), \(\ilreify(\tyI, \tmI)\), and
  \(\ilreflect(\tyI, \tm_{\neMark})\) are computed by case analysis on
  \(\tyI : \tytmstrI(\ty)\) and \(\tmI : \tyI(\tm)\).
\end{task}

\begin{construction}
  Let \(\mode\) a {\acrLAM}, \(\tytmstr\) be a type-term structure,
  and \(\nfalg\) a normalization predicate on \(\tytmstr\). We define
  a displayed type-term structure \(\nmtytmstr_{\mode, \nfalg}\) over
  \(\tytmstr\) as follows.
  \begin{align*}
    \begin{autobreak}
      \MoveEqLeft
      \ilTy^{\neord[1]}_{\nmtytmstr_{\mode, \nfalg}}(\ty) \defeq
      \ilExists_{\tyI : \ty \to \ilUniv_{\mode}}
      (\ilIsNfTy_{\nfalg}(\ty)
      \times (\ilForall_{\implicit{\tm : \ty}}\tyI(\tm) \to \ilIsNfTm_{\nfalg}(\tm))
      \times (\ilForall_{\implicit{\tm : \ty}}\ilIsNeTm_{\nfalg}(\tm) \to \tyI(\tm)))
    \end{autobreak}
    \\
    \begin{autobreak}
      \MoveEqLeft
      \ilTm^{\neord[1]}_{\nmtytmstr_{\mode, \nfalg}}(\ty, (\tyI, \dots), \tm) \defeq
      \tyI(\tm)
    \end{autobreak}
  \end{align*}
  By construction, \(\nmtytmstr_{\mode, \nfalg}\) is equipped with a
  canonical \(\nfalg\)-normalization structure. If \(\nfalg\) is
  \(\mode\)-modal, then \(\nmtytmstr_{\mode, \nfalg}\) is
  \(\mode\)-modal by the closure properties of modal types.
\end{construction}

\begin{task}
  \label{task:normalization-model}
  Let \(\mode\) be a {\acrLAM}, \(\tytmstr\) a \(\workingtth\)-type-term
  structure equipped with a variable structure, and \(\nfalg\) a
  \(\workingtth\)-normalization algebra on \(\tytmstr\) supporting
  variables. Extend \(\nmtytmstr_{\mode, \nfalg}\) to a displayed
  \(\workingtth\)-type-term structure over \(\tytmstr\) and prove that the
  \(\nfalg\)-normalization structure on \(\nmtytmstr_{\mode, \nfalg}\)
  supports \(\workingtth\).
\end{task}

\subsection{The normalization model}
\label{sec:normalization-model}

Suppose that we have done
\cref{task:normalization-structure,task:normalization-model}, and
assume
\cref{axm:working-logos-1-mode-sketch,axm:working-logos-1-type-term-structure,axm:working-logos-1-variable}
in type theory. Then \(\nmtytmstr_{\modeRen, \workingNfalg}\) is a
\(\modeRen\)-modal displayed \(\workingtth\)-type-term structure over
\(\workingTyTmStr\) whose \(\workingNfalg\)-normalization structure
supports \(\workingtth\) by \cref{task:normalization-model}. Let
\(\workingNmTyTmStr\) be the \((\modeRen \lor \modeSyn)\)-modal
\(\workingtth\)-type-term structure corresponding to
\((\workingTyTmStr, \nmtytmstr_{\modeRen, \workingNfalg})\) by
\cref{cor:fracture-and-gluing-type-term-structure}.

Going back outside, \(\workingNmTyTmStr\) corresponds to an
interpretation
\begin{math}
  \workingNmItpr_{\fun} : \workingtth \to \workingLogos_{\fun}.
\end{math}
We externalize \(\workingNmItpr_{\fun}\) keeping a connection to
\(\fun\). We have the full embedding
\(\workingRenIncNm_{\fun} : \Ctx(\modelI) \to \workingLogos_{\fun}\)
sending \(\ctx\) to the functor
\(\Ctx(\modelI)_{/ \ctx} \to \Ctx(\model)_{/ \fun_{\Ctx}(\ctx)}\) over
\(\fun_{\Ctx}\). Let
\(\Ctx(\workingNmModel_{\fun}) \subset \workingLogos_{\fun}\) be the
smallest full subcategory containing \(\Ctx(\modelI)\) and closed
under \(\workingNmItpr_{\fun}\)-context extension. Take the
externalization
\(\workingNmModel_{\fun} =
\extern{\workingNmItpr_{\fun}}_{\Ctx(\workingNmModel_{\fun})}\) and
call it the \emph{normalization model} associated to \(\fun\). The
projection \(\workingLogos_{\fun} \to \RFib_{\Ctx(\model)}\) restricts
to \(\Ctx(\workingNmModel_{\fun}) \to \Ctx(\model)\) and induces a
morphism
\begin{math}
  \workingNmProj_{\fun} : \workingNmModel_{\fun} \to \model
\end{math}
of models of \(\workingtth\). The base change along the inclusion
\(\workingRenIncNm_{\fun} : \Ctx(\modelI) \to
\Ctx(\workingNmModel_{\fun})\) induces a functor
\begin{math}
  \workingRenIncNm_{\fun}^{\pbMark} :
  \Fun(\neord[1], \RFib)_{(\workingNmProj_{\fun})_{\Ctx}}
  \to \workingLogos_{\fun}.
\end{math}
By construction, the composite
\begin{equation*}
  \workingtth
  \xrightarrow{\workingNmProj_{\fun}}
  \Fun(\neord[1], \RFib)_{(\workingNmProj_{\fun})_{\Ctx}}
  \xrightarrow{\workingRenIncNm_{\fun}^{\pbMark}}
  \workingLogos_{\fun}
\end{equation*}
is equivalent to
\(\workingNmItpr_{\fun} : \workingtth \to \workingLogos_{\fun}\). Furthermore,
by \cref{cst:var-to-term}, we have a map
\begin{equation*}
  \Ty(\workingNmItpr_{\fun}) \times_{\Ty(\model)} \Tm(\fun)
  \to \Tm(\workingNmItpr_{\fun})
\end{equation*}
in
\((\workingLogos_{\fun})_{/ \Ty(\workingNmItpr_{\fun}) \times_{\Ty(\model)}
  \Tm(\model)}\). This corresponds to a map
\begin{equation*}
  \workingVarToTm_{\fun} :
  \workingRenIncNm_{\fun}^{\pbMark} \Ty(\workingNmModel_{\fun})
  \times_{\fun_{\Ctx}^{\pbMark} \Ty(\model)}
  \Tm(\modelI)
  \to \workingRenIncNm_{\fun}^{\pbMark} \Tm(\workingNmModel_{\fun})
\end{equation*}
in
\((\RFib_{\Ctx(\modelI)})_{/ \workingRenIncNm_{\fun}^{\pbMark}
  \Ty(\workingNmModel_{\fun}) \times_{\fun_{\Ctx}^{\pbMark} \Ty(\model)}
  \fun_{\Ctx}^{\pbMark} \Tm(\model)}\).

\subsection{Running example}
\label{sec:function-types}

We consider the case when \(\workingtth = \basicTT_{\FunMark}\). We
first note that a \emph{displayed \(\basicTT_{\FunMark}\)-type-term
  structure on \(\tytmstr_{1}\)} is a displayed type-term structure
\(\tytmstr_{0}\) on \(\tytmstr_{1}\) equipped with the following
operators.
\begin{align*}
  \begin{autobreak}
    \MoveEqLeft
    \ilFun^{\neord[1]} :
    \ilForall_{\implicit{\ty_{1}}}
    \ilForall_{\ty_{0} : \tytmstr_{0}(\ty_{1})}
    \ilForall_{\implicit{\tyI_{1} : \ty_{1} \to \tytmstr_{1}}}
    (\ilForall_{\tm_{1}}\ty_{0}(\tm_{1}) \to \tytmstr_{0}(\tyI_{1}(\tm_{1})))
    \to \tytmstr_{0}(\ilFun(\ty_{1}, \tyI_{1}))
  \end{autobreak}
  \\
  \begin{autobreak}
    \MoveEqLeft
    \ilabs^{\neord[1]} :
    \ilForall_{\implicit{\ty_{1}}}
    \ilForall_{\implicit{\ty_{0} : \tytmstr_{0}(\ty_{1})}}
    \ilForall_{\implicit{\tyI_{1} : \ty_{1} \to \tytmstr_{1}}}
    \ilForall_{\implicit{\tyI_{0} : \ilForall_{\tm_{1}}\ty_{0}(\tm_{1}) \to \tytmstr_{0}(\tyI_{1}(\tm_{1}))}}
    \ilForall_{\implicit{\tmI_{1} : \ilForall_{\tm_{1}}\tyI_{1}(\tm_{1})}}
    (\ilForall_{\tm_{1}}\ilForall_{\tm_{0}}\tyI_{0}(\tm_{1}, \tm_{0})(\tmI_{1}(\tm_{1}))) \simeq \ilFun^{\neord[1]}(\ty_{0}, \tyI_{0})(\ilabs(\tmI_{1}))
  \end{autobreak}
  \\
  \begin{autobreak}
    \MoveEqLeft
    \ilNat^{\neord[1]} :
    \tytmstr_{0}(\ilNat)
  \end{autobreak}
  \\
  \begin{autobreak}
    \MoveEqLeft
    \ilzero^{\neord[1]} :
    \ilNat^{\neord[1]}(\ilzero)
  \end{autobreak}
  \\
  \begin{autobreak}
    \MoveEqLeft
    \ilsucc^{\neord[1]} :
    \ilForall_{\implicit{\nat_{1}}}
    \ilNat^{\neord[1]}(\nat_{1})
    \to \ilNat^{\neord[1]}(\ilsucc(\nat_{1}))
  \end{autobreak}
  \\
  \begin{autobreak}
    \MoveEqLeft
    \ilind_{\ilNat}^{\neord[1]} :
    \ilForall_{\implicit{\nat_{1}}}
    \ilForall_{\nat_{0} : \ilNat^{\neord[1]}(\nat_{1})}
    \ilForall_{\implicit{\ty_{1} : \ilNat \to \tytmstr_{1}}}
    \ilForall_{\ty_{0} : \ilForall_{\var_{1}} \ilNat^{\neord[1]}(\var_{1}) \to \tytmstr_{0}(\ty_{1}(\var_{1}))}
    \ilForall_{\implicit{\tm_{\zeroMark, 1}}}
    \ty_{0}(\blank, \ilzero^{\neord[1]}, \tm_{\zeroMark, 1})
    \to \ilForall_{\implicit{\tm_{\succMark, 1}}}
    (\ilForall_{\var_{1}}\ilForall_{\var_{0}}\ilForall_{\varI_{1}}\ty_{0}(\var_{1}, \var_{0}, \varI_{1}) \to \ty_{0}(\blank, \ilsucc^{\neord[1]}(\var_{0}), \tm_{\succMark, 1}(\var_{1}, \varI_{1})))
    \to \ty_{0}(\blank, \nat_{0}, \ilind_{\ilNat}(\nat_{1}, \ty_{1}, \tm_{\zeroMark, 1}, \tm_{\succMark, 1}))
  \end{autobreak}
  \\
  \begin{autobreak}
    \MoveEqLeft
    \ilbeta_{\ilzero}^{\neord[1]} :
    \text{(omitted)}
  \end{autobreak}
  \\
  \begin{autobreak}
    \MoveEqLeft
    \ilbeta_{\ilsucc}^{\neord[1]} :
    \text{(omitted)}
  \end{autobreak}
  \\
  \begin{autobreak}
    \MoveEqLeft
    \ilConstTy^{\neord[1]} :
    \tytmstr_{0}(\ilConstTy)
  \end{autobreak}
  \\
  \begin{autobreak}
    \MoveEqLeft
    \ilConstTyI^{\neord[1]} :
    \ilForall_{\implicit{\tm_{1}}}
    \ilConstTy^{\neord[1]}(\tm_{1})
    \to \tytmstr_{0}(\ilConstTyI(\tm_{1}))
  \end{autobreak}
  \\
  \begin{autobreak}
    \MoveEqLeft
    \ilConstTm^{\neord[1]} :
    \ilForall_{\implicit{\tm_{1}}}
    \ilForall_{\tm_{0} : \ilConstTy^{\neord[1]}(\tm_{1})}
    \ilConstTyI^{\neord[1]}(\tm_{0})(\ilConstTm(\tm_{1}))
  \end{autobreak}
\end{align*}

\paragraph{\Cref{task:normalization-structure}}

A \(\nfalg\)-normalization structure on a displayed
\(\basicTT_{\FunMark}\)-type-term structure \(\tytmstr_{0}\) over
\(\tytmstr_{1}\) \emph{supports \(\basicTT_{\FunMark}\)} if it is
equipped with identifications
\begin{align*}
  \begin{autobreak}
    \MoveEqLeft
    \ilnfty(\ilFun^{\neord[1]}(\ty_{0}, \tyI_{0})) =
    \ilFun_{\nftyMark}(\ilnfty(\ty_{0}), \ilAbs \var \var_{\varMark}. \ilnfty(\tyI_{0}(\ilVarToTm(\ty_{0}, \var_{\varMark}))))
  \end{autobreak}
  \\
  \begin{autobreak}
    \MoveEqLeft
    \ilreify(\blank, \ilabs^{\neord[1]}(\tmI_{0})) =
    \ilabs_{\nftyMark}(\ilAbs \var \var_{\varMark}. \ilreify(\blank, \tmI_{0}(\ilVarToTm(\ty_{0}, \var_{\varMark}))))
  \end{autobreak}
  \\
  \begin{autobreak}
    \MoveEqLeft
    \ilapp^{\neord[1]}(\ilreflect(\ilFun^{\neord[1]}(\ty_{0}, \tyI_{0}), \map_{\netmMark}), \tm_{0}) =
    \ilreflect(\blank, \ilapp_{\netmMark}(\map_{\netmMark}, \ilreify(\blank, \tm_{0})))
  \end{autobreak}
  \\
  \begin{autobreak}
    \MoveEqLeft
    \ilnfty(\ilNat^{\neord[1]}) =
    \ilNat_{\nftyMark}
  \end{autobreak}
  \\
  \begin{autobreak}
    \MoveEqLeft
    \ilreify(\blank, \ilzero^{\neord[1]}) =
    \ilzero_{\nftmMark}
  \end{autobreak}
  \\
  \begin{autobreak}
    \MoveEqLeft
    \ilreify(\blank, \ilsucc^{\neord[1]}(\nat_{0})) =
    \ilsucc_{\nftmMark}(\ilreify(\blank, \nat_{0}))
  \end{autobreak}
  \\
  \begin{autobreak}
    \MoveEqLeft
    \ilreify(\blank, \ilreflect(\ilNat^{\neord[1]}, \nat_{\netmMark})) =
    \ilNeToNf_{\ilNat}(\nat_{\netmMark})
  \end{autobreak}
  \\
  \begin{autobreak}
    \MoveEqLeft
    \ilind_{\ilNat}^{\neord[1]}(\ilreflect(\ilNat^{\neord[1]}, \nat_{\netmMark}), \tyX_{0}, \tmX_{\zeroMark, 0}, \tmX_{\succMark, 0}) =
    \ilreflect(\blank,
    \ilind_{\ilNat, \netmMark}(\nat_{\netmMark},
    \ilAbs \var \var_{\varMark}. \ilnfty(\tyX_{0}(\ilVarToTm(\blank, \var_{\varMark}))),
    \ilreify(\blank, \tmX_{\zeroMark, 0}),
    \ilAbs \var \var_{\varMark} \varI \varI_{\varMark}. \ilreify(\blank, \tmX_{\succMark, 0}(\ilVarToTm(\blank, \var_{\varMark}), \ilVarToTm(\blank, \varI_{\varMark})))))
  \end{autobreak}
  \\
  \begin{autobreak}
    \MoveEqLeft
    \ilnfty(\ilConstTy^{\neord[1]}) =
    \ilConstTy_{\nftyMark}
  \end{autobreak}
  \\
  \begin{autobreak}
    \MoveEqLeft
    \ilreify(\blank, \ilreflect(\ilConstTy^{\neord[1]}, \tmII_{\netmMark})) =
    \ilNeToNf_{\ilConstTy}(\tmII_{\netmMark})
  \end{autobreak}
  \\
  \begin{autobreak}
    \MoveEqLeft
    \ilnfty(\ilConstTyI^{\neord[1]}(\tmII_{0})) =
    \ilConstTyI_{\nftyMark}(\ilreify(\blank, \tmII_{0}))
  \end{autobreak}
  \\
  \begin{autobreak}
    \MoveEqLeft
    \ilreify(\blank, \ilreflect(\ilConstTyI^{\neord[1]}(\tmII_{0}), \tmIII_{\netmMark})) =
    \ilNeToNf_{\ilConstTyI}(\ilreify(\blank, \tmII_{0}), \tmIII_{\netmMark})
  \end{autobreak}
  \\
  \begin{autobreak}
    \MoveEqLeft
    \ilConstTm^{\neord[1]}(\tmII_{0}) =
    \ilreflect(\blank, \ilConstTm_{\netmMark}(\ilreify(\blank, \tmII_{0})))
  \end{autobreak}
\end{align*}
over the context
\begin{align*}
  \begin{autobreak}
    \MoveEqLeft[0]
    \ty_{1} : \tytmstr_{1},
    \ty_{0} : \tytmstr_{0}(\ty_{1}),
    \tyI_{1} : \ty_{1} \to \tytmstr_{1},
    \tyI_{0} : \ilForall_{\tm_{1}}\ty_{0}(\tm_{1}) \to \tytmstr_{0}(\tyI_{1}(\tm_{1})),
    \tmI_{1} : \ilForall_{\tm_{1}}\tyI_{1}(\tm_{1}),
    \tmI_{0} : \ilForall_{\tm_{1}}\ilForall_{\tm_{1}}\tyI_{0}(\tm_{1}, \tm_{0})(\tmI_{1}(\tm_{1})),
    \map : \ilFun(\ty_{1}, \tyI_{1}),
    \map_{\netmMark} : \ilIsNeTm_{\nfalg}(\map),
    \tm_{1} : \ty_{1},
    \tm_{0} : \ty_{0}(\tm_{1}),
    \nat_{1} : \ilNat,
    \nat_{0} : \ilNat^{\neord[1]}(\nat_{1}),
    \nat_{\netmMark} : \ilIsNeTm_{\nfalg}(\nat_{1}),
    \tyX_{1} : \ilNat \to \tytmstr_{1},
    \tyX_{0} : \ilForall_{\implicit{\var}}\ilNat^{\neord[1]}(\var) \to \tytmstr_{0}(\tyX_{1}(\var)),
    \tmX_{\zeroMark, 1} : \tyX_{1}(\ilzero),
    \tmX_{\zeroMark, 0} : \tyX_{0}(\ilzero^{\neord[1]})(\tmX_{\zeroMark, 1}),
    \tmX_{\succMark, 1} : \ilForall_{\var_{1}}\tyX_{1}(\var_{1}) \to \tyX_{1}(\ilsucc(\var_{1})),
    \tmX_{\succMark, 0} : \ilForall_{\implicit{\var_{1}}}
    \ilForall_{\var_{0} : \ilNat^{\neord[1]}(\var_{1})}
    \ilForall_{\implicit{\varI_{1}}}
    \tyX_{0}(\var_{0})(\varI_{1}) \to \tyX_{0}(\ilsucc(\var_{0}))(\tmX_{\succMark, 1}(\var_{1}, \varI_{1})),
    \tmII_{1} : \ilConstTy,
    \tmII_{\netmMark} : \ilIsNeTm_{\nfalg}(\tmII_{1}),
    \tmII_{0} : \ilConstTy^{\neord[1]}(\tmII_{1}),
    \tmIII_{1} : \ilConstTyI(\tmII_{1}),
    \tmIII_{\netmMark} : \ilIsNeTm_{\nfalg}(\tmIII_{1}).
  \end{autobreak}
\end{align*}

\paragraph{\Cref{task:normalization-model}}

For function types, let \(\ty_{1} : \tytmstr\),
\(\ty_{0} : \nmtytmstr_{\mode, \nfalg}(\ty_{1})\),
\(\tyI_{1} : \ty_{1} \to \tytmstr\), and
\(\tyI_{0} : \ilForall_{\tm_{1} : \ty_{1}}\ty_{0}(\tm_{1}) \to
\nmtytmstr_{\mode, \nfalg}(\tyI_{1}(\tm_{1}))\). By the definition of
\(\nmtytmstr_{\mode, \nfalg}\), the element
\(\ilFun^{\neord[1]}(\ty_{0}, \tyI_{0}) : \nmtytmstr_{\mode,
  \nfalg}(\ilFun(\ty_{1}, \tyI_{1}))\) must be a type family
\(\ilFun(\ty_{1}, \tyI_{1}) \to \ilUniv_{\mode}\) equipped with other
components \(\ilnfty(\ilFun^{\neord[1]}(\ty_{0}, \tyI_{0}))\),
\(\ilreify(\ilFun^{\neord[1]}(\ty_{0}, \tyI_{0}))\), and
\(\ilreflect(\ilFun^{\neord[1]}(\ty_{0}, \tyI_{0}))\). For
\(\ilFun^{\neord[1]}(\ty_{0}, \tyI_{0})\) to be the function type, its
underlying type family
\(\ilFun(\ty_{1}, \tyI_{1}) \to \ilUniv_{\mode}\) must send
\(\ilabs(\tmI_{1})\) for
\(\tmI_{1} : \ilForall_{\tm_{1}}\tyI_{1}(\tm_{1})\) to
\begin{math}
  \ilForall_{\tm_{1}}\ilForall_{\tm_{0}}\tyI_{0}(\tm_{1}, \tm_{0})(\tmI_{1}(\tm_{0})) : \ilUniv_{\mode},
\end{math}
and there exists a unique such type family since \(\ilabs\) is an
equivalence. The other components are also uniquely determined by the
required identifications.

The underlying type of \(\ilNat^{\neord[1]}\) is defined as a
\(\mode\)-modal inductive type family
\(\ilNat^{\neord[1]} : \ilNat \to \ilUniv_{\mode}\) with an additional
constructor for neutral terms.
\begin{align*}
  \begin{autobreak}
    \MoveEqLeft
    \ilzero^{\neord[1]} :
    \ilNat^{\neord[1]}(\ilzero)
  \end{autobreak}
  \\
  \begin{autobreak}
    \MoveEqLeft
    \ilsucc^{\neord[1]} :
    \ilForall_{\implicit{\nat}}
    \ilNat^{\neord[1]}(\nat)
    \to \ilNat^{\neord[1]}(\ilsucc(\nat))
  \end{autobreak}
  \\
  \begin{autobreak}
    \MoveEqLeft
    \ilreflect(\ilNat^{\neord[1]}, \blank) :
    \ilForall_{\implicit{\nat}}
    \ilIsNeTm_{\nfalg}(\nat)
    \to \ilNat^{\neord[1]}(\nat)
  \end{autobreak}
\end{align*}
\(\ilnfty(\ilNat^{\neord[1]})\) must be \(\ilNat_{\nftyMark}\). The
required identifications for \(\ilreify(\ilNat^{\neord[1]}, \blank)\)
explain its definition by induction. The required identification for
\(\ilind_{\ilNat, \netmMark}\) explains how the induction operator
\(\ilind_{\ilNat}^{\neord[1]}\) is defined on the constructor
\(\ilreflect(\ilNat^{\neord[1]}, \blank)\).

For the constant \(\ilConstTy^{\neord[1]}\), we define its underlying
type family to be \(\ilAbs
\tmII_{1}. \ilIsNeTm_{\nfalg}(\tmII_{1})\). We must define
\(\ilnfty(\ilConstTy^{\neord[1]}) \defeq \ilConstTy_{\nftyMark}\). We
define
\(\ilreify(\ilConstTy^{\neord[1]}, \tmII_{\netmMark}) \defeq
\ilNeToNf_{\ilConstTy}(\tmII_{\netmMark})\) and
\(\ilreflect(\ilConstTy^{\neord[1]}, \tmII_{\netmMark}) \defeq
\tmII_{\netmMark}\), and then the required identifications are
satisfied. The constant \(\ilConstTyI^{\neord[1]}\) is defined in the
same way. We must define \(\ilConstTm^{\neord[1]}(\tmII_{0})\) to be
\(\ilreflect(\blank, \ilConstTm_{\netmMark}(\ilreify(\blank,
\tmII_{0})))\).

\section{Normalization results}
\label{sec:norm-results}

We consider the special case of \cref{sec:normalization-models} where
\(\model\) is the initial model \(\workingInitModel\) of
\(\workingtth\) and \(\fun\) is the initial renaming model
\(\workingRenCounit : \workingRenModel \to \workingInitModel\). In this
case, we omit the subscript \({}_{\workingRenCounit}\) and simply
write \(\workingLogos = \workingLogos_{\workingRenCounit}\),
\(\workingNmModel = \workingNmModel_{\workingRenCounit}\), and so
on. The goal of this section is to show that \(\workingLogos\)
validates the following propositions.
\begin{align}
  \label{eq:isnfty-contractible}
  \begin{autobreak}
    \MoveEqLeft
    \ilForall_{\ty : \workingTyTmStr}
    \ilIsContr(\ilIsNfTy_{\workingNfalg}(\ty))
  \end{autobreak}
  \\
  \label{eq:isnftm-contractible}
  \begin{autobreak}
    \MoveEqLeft
    \ilForall_{\ty : \workingTyTmStr}
    \ilForall_{\tm : \ty}
    \ilIsContr(\ilIsNfTm_{\workingNfalg}(\tm))
  \end{autobreak}
\end{align}
That is, every type or term has a unique normal form. Note that the
initiality of \(\workingInitModel\) or \(\workingRenModel\) is not
visible in the internal language of \(\workingLogos\). Therefore, we
have to do some external reasoning first and then embed
\(\workingLogos\) into another \(\infty\)-logos to internalize necessary
ingredients.

\subsection{Semantic setup}
\label{sec:semantic-setup-1}

By the initiality of \(\workingInitModel\), we have a unique section
\(\workingSection : \workingInitModel \to \workingNmModel\) of
\(\workingNmProj : \workingNmModel \to \workingInitModel\). The
\emph{relative induction principle} of \textcite{bocquet2021relative}
ensures that \(\workingSection\) is also related to
\(\workingRenCounit : \workingRenModel \to \workingInitModel\) and
\(\workingRenIncNm : \Ctx(\workingRenModel) \to \Ctx(\workingNmModel)\).

\begin{proposition}
  \label{prop:relative-induction-trans}
  One can construct a natural transformation
  \begin{equation*}
    \workingTrans : \workingRenIncNm \To \workingSection_{\Ctx} \comp
    \workingRenCounit_{\Ctx} : \Ctx(\workingRenModel) \to \Ctx(\workingNmModel)
  \end{equation*}
  over \(\Ctx(\workingInitModel)\). Moreover, the composite
  \begin{align*}
    \begin{autobreak}
      \MoveEqLeft
      \Tm(\workingRenModel)
      \xrightarrow{\workingRenCounit_{\Tm}} \workingRenCounit_{\Ctx}^{\pbMark} \Tm(\workingInitModel)
      \xrightarrow{\workingRenCounit_{\Ctx}^{\pbMark} \workingSection_{\Tm}} \workingRenCounit_{\Ctx}^{\pbMark} \workingSection_{\Ctx}^{\pbMark} \Tm(\workingNmModel)
      \xrightarrow{\workingTrans^{\pbMark}} \workingRenIncNm^{\pbMark} \Tm(\workingNmModel)
    \end{autobreak}
  \end{align*}
  is equivalent over
  \(\workingRenIncNm^{\pbMark} \Ty(\workingNmModel)
  \times_{\workingRenCounit_{\Ctx}^{\pbMark} \Ty(\workingInitModel)}
  \workingRenCounit_{\Ctx}^{\pbMark} \Tm(\workingInitModel)\) to the
  composite
  \begin{align*}
    \begin{autobreak}
      \MoveEqLeft
      \Tm(\workingRenModel)
      \xrightarrow{(\workingTrans^{\pbMark} \comp \workingRenCounit_{\Ctx}^{\pbMark} \workingSection_{\Ty}, \id)} \workingRenIncNm^{\pbMark} \Ty(\workingNmModel) \times_{\workingRenCounit_{\Ctx}^{\pbMark} \Ty(\workingInitModel)} \Tm(\workingRenModel)
      \xrightarrow{\workingVarToTm} \workingRenIncNm^{\pbMark} \Tm(\workingNmModel).
    \end{autobreak}
  \end{align*}
\end{proposition}

The proof of \cref{prop:relative-induction-trans} is essentially the
same as the \(1\)-categorical case \parencite{bocquet2021relative} and
given in \cref{sec:proof-relative-induction}.

We now have the following diagram.
\begin{equation}
  \label{eq:working-logos-1-diagram}
  \begin{tikzcd}[column sep=0ex]
    & \Ctx(\workingNmModel)
    \arrow[dr, equal] \\
    \Ctx(\workingRenModel)
    \arrow[ur, "\workingRenIncNm"]
    \arrow[rr, "\workingSection_{\Ctx} \comp \workingRenCounit_{\Ctx}"{description},
    "{\phantom{a}}"{name = a0}, "{\phantom{a}}"'{name = a1}]
    \arrow[dr, "\workingRenCounit_{\Ctx}"'] &
    \arrow[from = u, to = a0, To, "\workingTrans",
    start anchor = {[yshift = -1ex]}]
    \arrow[from = a1, to = d, phantom, "\simeq"{description}] &
    \Ctx(\workingNmModel)
    \arrow[r, "\workingNmProj_{\Ctx}"] &
    [6ex]
    \Ctx(\workingInitModel) \\
    & \Ctx(\workingInitModel)
    \arrow[ur, "\workingSection_{\Ctx}"']
  \end{tikzcd}
\end{equation}
Let \(\workingMSI\) be the mode sketch
\begin{equation*}
  \begin{tikzcd}[row sep = 3ex]
    & 2'
    \arrow[dr] \\
    0
    \arrow[ur]
    \arrow[rr, "\phantom{a}"{name = a0},
    "\phantom{a}"'{name = a1}]
    \arrow[dr] &
    \arrow[from = u, to = a0, To, start anchor = {[yshift = -1ex]}]
    \arrow[from = d, to = a1, phantom, "\simeq"{description}, end anchor = {[yshift = 1ex]}]&
    2
    \arrow[r] &
    1 \\
    & 3
    \arrow[ur]
  \end{tikzcd}
\end{equation*}
where only \((0 < 2' < 2)\) is not thin. Then
\cref{eq:working-logos-1-diagram} is a functor
\(\workingLogosIBase : \realize{\workingMSI} \to \Cat^{\nMark{2}}\), and
our second working \(\infty\)-logos \(\workingLogosI\) is
\(\Fun(\realize{\workingMSI}, \RFib)_{\workingLogosIBase}\).

\subsection{Axiomatic setup}
\label{sec:axiomatic-setup-1}

\begin{axiom}
  \label{axm:working-logos-2-mode-sketch}
  A function \(\mode : \workingMSI \to \ilLAM\) satisfying
  \cref{axm:mode-sketch-orthogonal,axm:mode-sketch-thin} is
  postulated. We define \(\modeRen \defeq \mode(0)\),
  \(\modeSyn \defeq \mode(1)\), \(\modeNm \defeq \mode(2)\),
  \(\modeNmI \defeq \mode(2')\), and \(\modeSect \defeq \mode(3)\).
\end{axiom}

By construction, the {\acrLAM} \(\modeRen \lor \modeSyn\) corresponds
to the previous \(\infty\)-logos \(\workingLogos\). Therefore, we may
assume all the previous results.

\begin{axiom}
  \label{axm:working-logos-2-axioms-of-working-logos-1}
  Presuppose
  \cref{axm:working-logos-2-mode-sketch}. \Cref{axm:working-logos-1-type-term-structure,axm:working-logos-1-variable}
  are assumed.
\end{axiom}

\begin{axiom}
  \label{axm:working-logos-2-results-of-working-logos-1}
  Presuppose
  \cref{axm:working-logos-2-mode-sketch,axm:working-logos-2-axioms-of-working-logos-1}. A
  \(\modeRen\)-modal displayed \(\workingtth\)-type-term structure
  \(\workingNmDTyTmStr\) over \(\workingTyTmStr\) is postulated, and
  it is assumed to have a \(\workingNfalg\)-normalization structure
  supporting \(\workingtth\).
\end{axiom}

Corresponding to the diagram
\begin{math}
  \workingInitModel
  \xrightarrow{\workingSection}
  \workingNmModel
  \xrightarrow{\workingNmProj}
  \workingInitModel,
\end{math}
we have a \((\modeSect \lor \modeNm \lor \modeSyn)\)-modal
\(\workingtth\)-type-term structure whose \(\modeSyn\)-component is
\(\workingTyTmStr\). Since
\(\workingNmProj \comp \workingSection \simeq \id\), the
\((\modeSect \lor \modeSyn)\)-component of the \(\workingtth\)-type-term
structure is \(\modeSyn\)-modal. Furthermore, the functor
\(\opMode^{\modeSect \lor \modeNm \lor \modeSyn}_{\modeSect} :
\ilUniv_{\modeSect \lor \modeNm \lor \modeSyn} \to \ilUniv_{\modeSect}\)
preserves dependent function types indexed over the terms in the
\(\workingtth\)-type-term structure by \cref{rem:modal-tininess}. It is then
fractured into the following structures
(\cref{axm:working-logos-2-normalization-model,axm:working-logos-2-section,axm:working-logos-2-tiny}).

\begin{axiom}
  \label{axm:working-logos-2-normalization-model}
  Presuppose
  \cref{axm:working-logos-2-mode-sketch,axm:working-logos-2-axioms-of-working-logos-1}. A
  \(\modeNm\)-modal displayed \(\workingtth\)-type-term structure
  \(\workingNmModelIn\) over \(\workingTyTmStr\) is postulated.
\end{axiom}

\begin{definition}
  Let \(\tytmstrI\) be a displayed \(\workingtth\)-type-term structure over a
  \(\workingtth\)-type-term structure \(\tytmstr\). The notion of a
  \emph{presection of \(\tytmstrI\)} is defined by the logical
  relation translation of the notion of a \(\workingtth\)-type-term structure
  with respect to the inverse category \(\neord[2]\). Concretely, a
  presection \(\sect\) of \(\tytmstrI\) consists of
  \begin{align*}
    \begin{autobreak}
      \MoveEqLeft
      \ilTy^{\neord[2]}_{\sect} :
      \ilForall_{\implicit{\ty}}
      \tytmstrI(\ty)
      \to \ilUniv
    \end{autobreak}
    \\
    \begin{autobreak}
      \MoveEqLeft
      \ilTm^{\neord[2]}_{\sect} :
      \ilForall_{\implicit{\ty}}
      \ilForall_{\implicit{\tyI : \tytmstrI(\ty)}}
      \ilTy^{\neord[2]}_{\sect}(\tyI)
      \to \ilForall_{\implicit{\tm}}
      \tyI(\tm)
      \to \ilUniv
    \end{autobreak}
  \end{align*}
  and operators depending on the presentation of \(\workingtth\). We regard
  \(\sect \mapsto \ilTy^{\neord[2]}_{\sect}\) and
  \(\ilTm^{\neord[2]}_{\sect}\) as implicit coercions.
\end{definition}

\begin{definition}
  \label{def:section-of-tytmstr}
  Let \(\mode_{0}\), \(\mode_{1}\), \(\mode_{2}\) be {\acrLAMs} such
  that \(\mode_{\idx} \le {}^{\orthMark} \mode_{\idxI}\) for all
  \(\idx < \idxI\). Let \(\tytmstr\) be a \(\mode_{2}\)-modal
  type-term structure, \(\tytmstrI\) a \(\mode_{1}\)-modal displayed
  type-term structure over \(\tytmstr\), and \(\sect\) a
  \(\mode_{0}\)-modal presection of \(\tytmstrI\). We say \(\sect\) is
  a \emph{section} if the type
  \(\opMode_{\mode_{0}}(\ilExists_{\tyI :
    \tytmstrI(\ty)}\sect(\tyI))\) is contractible for any
  \(\ty : \tytmstr\) and if the type
  \(\opMode_{\mode_{0}}(\ilExists_{\tmI : \tyI(\tm)}\tyII(\tmI))\) is
  contractible for any \(\tyI : \tytmstrI(\ty)\),
  \(\tyII : \sect(\tyI)\), and \(\tm : \ty\).
\end{definition}

\begin{axiom}
  \label{axm:working-logos-2-section}
  Presuppose
  \cref{axm:working-logos-2-mode-sketch,axm:working-logos-2-axioms-of-working-logos-1,axm:working-logos-2-normalization-model}. A
  \(\modeSect\)-modal section \(\workingSectionIn\) of
  \(\workingNmModelIn\) is postulated.
\end{axiom}

\begin{axiom}
  \label{axm:working-logos-2-tiny}
  Presuppose
  \cref{axm:working-logos-2-mode-sketch,axm:working-logos-2-axioms-of-working-logos-1,axm:working-logos-2-normalization-model,axm:working-logos-2-section}. For
  any \(\ty : \workingTyTmStr\), \(\tyI : \workingNmModelIn(\ty)\),
  \(\tyII : \workingSectionIn(\tyI)\), and
  \begin{math}
    \tyX : \ilForall_{\tm : \ty}\ilForall_{\tmI : \tyI(\tm)}\tyII(\tmI) \to \ilUniv_{\modeSect \lor \modeNm \lor \modeSyn},
  \end{math}
  the canonical map
  \begin{equation*}
    \opMode_{\modeSect}(\ilForall_{\tm}\ilForall_{\tmI}\ilForall_{\tmII}\tyX(\tm, \tmI, \tmII))
    \to (\ilForall_{\tm}\ilForall_{\tmI}\ilForall_{\tmII}\opMode_{\modeSect} \tyX(\tm, \tmI, \tmII))
  \end{equation*}
  is assumed to be an equivalence.
\end{axiom}

\begin{construction}
  \label{cst:section-type-1}
  Let \(\sect\) be a section as in \cref{def:section-of-tytmstr}. We
  have the actions of \(\sect\) on types and terms
  \begin{align*}
    \begin{autobreak}
      \MoveEqLeft
      \ilSectApp_{\ilTy_{1}}(\sect) :
      \ilForall_{\ty : \tytmstr}
      \opMode_{\mode_{0}} \tytmstrI(\ty)
    \end{autobreak}
    \\
    \begin{autobreak}
      \MoveEqLeft
      \ilSectApp_{\ilTm_{1}}(\sect) :
      \ilForall_{\implicit{\ty : \tytmstr}}
      \ilForall_{\tm : \ty}
      \ilSectApp_{\ilTy_{1}}(\sect, \ty)(\tm)
    \end{autobreak}
  \end{align*}
  by taking the center of contraction, where
  \(\ilTm^{\neord[1]}_{\tytmstrI}(\ty) : \tytmstrI(\ty) \to \ty \to
  \ilUniv\) induces a function
  \(\opMode_{\mode_{0}} \tytmstrI(\ty) \to \ty \to \ilUniv_{\mode_{0}}\)
  by which we regard \(\ilSectApp_{\ilTm_{1}}(\sect, \ty)\) as a
  function \(\ty \to \ilUniv_{\mode_{0}}\).
\end{construction}

To axiomatize the equivalence
\begin{math}
  \workingRenIncNm^{\pbMark} \workingNmProj \simeq
  \workingNmItpr,
\end{math}
recall from \cref{exm:cartesian-natural-transformation-internal} that
natural transformations cartesian on representable maps correspond to
cartesian morphisms between type-term structures. Cartesian morphisms
between displayed type-term structures are similarly defined. Since a
natural equivalence is in particular cartesian on representable maps,
we have the following.

\begin{axiom}
  \label{axm:working-logos-2-nm-comparison-1}
  Presuppose
  \cref{axm:working-logos-2-mode-sketch,axm:working-logos-2-axioms-of-working-logos-1,axm:working-logos-2-normalization-model}. A
  \(\modeNmI\)-modal displayed \(\workingtth\)-type-term structure
  \(\workingNmModelInI\) over \(\workingTyTmStr\) and a cartesian
  morphism \(\workingCorr\) from \(\workingNmModelIn\) to
  \(\workingNmModelInI\) are postulated.
\end{axiom}

\begin{axiom}
  \label{axm:working-logos-2-nm-comparison-2}
  Presuppose
  \cref{axm:working-logos-2-mode-sketch,axm:working-logos-2-axioms-of-working-logos-1,axm:working-logos-2-results-of-working-logos-1,axm:working-logos-2-normalization-model,axm:working-logos-2-nm-comparison-1}. A
  cartesian morphism \(\workingCorrI\) from \(\workingNmModelInI\) to
  \(\workingNmDTyTmStr\) is postulated.
\end{axiom}

\begin{remark}
  \(\workingCorr\) is interpreted as the identity on
  \(\workingNmModel\). This is introduced merely because the language
  of mode sketch is restrictive.
\end{remark}

We can now internally construct a section of \(\workingNmDTyTmStr\).

\begin{construction}
  \label{cst:working-logos-2-section-to-nn}
  Assume
  \cref{axm:working-logos-2-mode-sketch,axm:working-logos-2-axioms-of-working-logos-1,axm:working-logos-2-results-of-working-logos-1,axm:working-logos-2-normalization-model,axm:working-logos-2-section,axm:working-logos-2-nm-comparison-1,axm:working-logos-2-nm-comparison-2}. For
  any \(\ty : \workingTyTmStr\), we define
  \begin{math}
    \ilSectApp_{\ilTy_{1}}^{\modeRen}(\workingSectionIn, \ty)
    : \workingNmDTyTmStr(\ty)
  \end{math}
  by
  \begin{align*}
    \begin{autobreak}
      \ilUnit
      \xrightarrow{\ilSectApp_{\ilTy_{1}}(\workingSectionIn, \ty)} \opMode^{\modeNm}_{\modeSect} \workingNmModelIn(\ty)
      \xrightarrow{\unitMode_{\modeRen}} \opMode^{\modeSect}_{\modeRen} \opMode^{\modeNm}_{\modeSect} \workingNmModelIn(\ty)
      \simeq \opMode^{\modeNm}_{\modeRen} \workingNmModelIn(\ty)
      \xrightarrow{\unitMode^{\modeRen; \modeNm}_{\modeNmI}} \opMode^{\modeNmI}_{\modeRen} \opMode^{\modeNm}_{\modeNmI} \workingNmModelIn(\ty)
      \xrightarrow{\opMode^{\modeNmI}_{\modeRen} \ilCorrApp_{\ilTy_{1}}(\workingCorr)} \opMode^{\modeNmI}_{\modeRen} \workingNmModelInI(\ty)
      \xrightarrow{\ilCorrApp_{\ilTy_{1}}(\workingCorrI)} \workingNmDTyTmStr(\ty).
    \end{autobreak}
  \end{align*}
  We can similarly define
  \(\ilSectApp_{\ilTm_{1}}^{\modeRen}(\workingSectionIn, \tm) :
  \ilSectApp_{\ilTy_{1}}^{\modeRen}(\workingSectionIn, \ty)(\tm)\) for
  any \(\tm : \ty\).
\end{construction}

Finally, we internalize the second half of
\cref{prop:relative-induction-trans}.

\begin{axiom}
  \label{axm:working-logos-2-var-to-tm}
  Presuppose
  \cref{axm:working-logos-2-mode-sketch,axm:working-logos-2-axioms-of-working-logos-1,axm:working-logos-2-results-of-working-logos-1,axm:working-logos-2-normalization-model,axm:working-logos-2-section,axm:working-logos-2-nm-comparison-1,axm:working-logos-2-nm-comparison-2}. An
  identification
  \begin{math}
    \ilSectApp_{\ilTm_{1}}^{\modeRen}(\workingSectionIn, \var) = \ilVarToTm(\blank, \var_{\varMark})
  \end{math}
  is postulated for any \(\ty : \workingTyTmStr\), \(\var : \ty\), and
  \(\var_{\varMark} : \ilIsVar(\var)\).

\end{axiom}

\subsection{Normalization functions}
\label{sec:norm-funct}

We work in type theory.

\begin{construction}
  Assume
  \cref{axm:working-logos-2-mode-sketch,axm:working-logos-2-axioms-of-working-logos-1,axm:working-logos-2-results-of-working-logos-1,axm:working-logos-2-normalization-model,axm:working-logos-2-section,axm:working-logos-2-nm-comparison-1,axm:working-logos-2-nm-comparison-2}. We
  define \emph{normalization functions}
  \begin{align*}
    \begin{autobreak}
      \MoveEqLeft
      \ilNormalize_{\ilTy} :
      \ilForall_{\ty : \workingTyTmStr}
      \ilIsNfTy_{\workingNfalg}(\ty)
    \end{autobreak}
    \\
    \begin{autobreak}
      \MoveEqLeft
      \ilNormalize_{\ilTm} :
      \ilForall_{\implicit{\ty : \workingTyTmStr}}
      \ilForall_{\tm : \ty}
      \ilIsNfTm_{\workingNfalg}(\tm)
    \end{autobreak}
    \\
    \begin{autobreak}
      \MoveEqLeft
      \ilNormalize_{\ilTy}(\ty) \defeq
      \ilnfty(\ilSectApp_{\ilTy_{1}}^{\modeRen}(\workingSectionIn, \ty))
    \end{autobreak}
    \\
    \begin{autobreak}
      \MoveEqLeft
      \ilNormalize_{\ilTm}(\tm) \defeq
      \ilreify(\blank, \ilSectApp_{\ilTm_{1}}^{\modeRen}(\workingSectionIn, \tm)).
    \end{autobreak}
  \end{align*}
\end{construction}

\begin{construction}
  Assume
  \cref{axm:working-logos-2-mode-sketch,axm:working-logos-2-axioms-of-working-logos-1,axm:working-logos-2-results-of-working-logos-1,axm:working-logos-2-normalization-model,axm:working-logos-2-section,axm:working-logos-2-nm-comparison-1,axm:working-logos-2-nm-comparison-2}. We
  define predicates
  \begin{align*}
    \begin{autobreak}
      \MoveEqLeft
      \ilIsNfTy^{\neord[1]}_{\workingUniqPred} :
      \ilForall_{\implicit{\ty : \workingTyTmStr}}
      \ilIsNfTy_{\workingNfalg}(\ty)
      \to \ilUniv_{\modeRen}
    \end{autobreak}
    \\
    \begin{autobreak}
      \MoveEqLeft
      \ilIsNfTm^{\neord[1]}_{\workingUniqPred} :
      \ilForall_{\implicit{\ty : \workingTyTmStr}}
      \ilForall_{\implicit{\tm : \ty}}
      \ilIsNfTm_{\workingNfalg}(\tm)
      \to \ilUniv_{\modeRen}
    \end{autobreak}
    \\
    \begin{autobreak}
      \MoveEqLeft
      \ilIsNeTm^{\neord[1]}_{\workingUniqPred} :
      \ilForall_{\implicit{\ty : \workingTyTmStr}}
      \ilForall_{\implicit{\tm : \ty}}
      \ilIsNeTm_{\workingNfalg}(\tm)
      \to \ilUniv_{\modeRen}
    \end{autobreak}
    \\
    \begin{autobreak}
      \MoveEqLeft
      \ilIsNfTy^{\neord[1]}_{\workingUniqPred}(\ty_{\nftyMark}) \defeq
      (\ilNormalize_{\ilTy}(\ty) = \ty_{\nftyMark})
    \end{autobreak}
    \\
    \begin{autobreak}
      \MoveEqLeft
      \ilIsNfTm^{\neord[1]}_{\workingUniqPred}(\tm_{\nftmMark}) \defeq
      (\ilNormalize_{\ilTm}(\tm) = \tm_{\nftmMark})
    \end{autobreak}
    \\
    \begin{autobreak}
      \MoveEqLeft
      \ilIsNeTm^{\neord[1]}_{\workingUniqPred}(\tm_{\netmMark}) \defeq
      (\ilSectApp_{\ilTm_{1}}^{\modeRen}(\workingSectionIn, \tm) = \ilreflect(\blank, \tm_{\netmMark})).
    \end{autobreak}
  \end{align*}
\end{construction}

We want to show that \(\ilIsNfTy^{\neord[1]}_{\workingUniqPred}\),
\(\ilIsNfTm^{\neord[1]}_{\workingUniqPred}\), and
\(\ilIsNeTm^{\neord[1]}_{\workingUniqPred}\) are part of a
\emph{displayed \(\workingtth\)-normal form algebra \(\workingUniqPred\) over
  \(\workingNfalg\) supporting variables}, which is defined by the
logical relation translation with respect to the inverse category
\(\neord[1]\), so that it admits a section by the initiality of
\(\workingNfalg\).

\begin{proposition}
  \label{prop:uniqueness-of-nf-var}
  Assume
  \cref{axm:working-logos-2-mode-sketch,axm:working-logos-2-axioms-of-working-logos-1,axm:working-logos-2-results-of-working-logos-1,axm:working-logos-2-normalization-model,axm:working-logos-2-section,axm:working-logos-2-nm-comparison-1,axm:working-logos-2-nm-comparison-2,axm:working-logos-2-var-to-tm}. Then
  \(\workingUniqPred\) supports variables in the sense that we have a
  function
  \begin{align*}
    \begin{autobreak}
      \MoveEqLeft
      \ilForall_{\ty : \workingTyTmStr}
      \ilForall_{\var : \ty}
      \ilForall_{\var_{\varMark} : \ilIsVar(\var)}
      \ilIsNeTm^{\neord[1]}_{\workingUniqPred}(\ilVarToNe(\var_{\varMark})).
    \end{autobreak}
  \end{align*}
\end{proposition}
\begin{proof}
  By the definition of \(\ilVarToTm\) and
  \cref{axm:working-logos-2-var-to-tm}.
\end{proof}

\begin{task}
  \label{task:uniqueness-of-normal-forms}
  Assume
  \cref{axm:working-logos-2-mode-sketch,axm:working-logos-2-axioms-of-working-logos-1,axm:working-logos-2-results-of-working-logos-1,axm:working-logos-2-normalization-model,axm:working-logos-2-section,axm:working-logos-2-tiny,axm:working-logos-2-nm-comparison-1,axm:working-logos-2-nm-comparison-2,axm:working-logos-2-var-to-tm}. Show
  that \(\workingUniqPred\) is a displayed \(\workingtth\)-normal form
  algebra over \(\workingNfalg\).
\end{task}

\begin{remark}
  We have not used \cref{axm:working-logos-2-tiny} explicitly. This
  axiom is needed in \cref{task:uniqueness-of-normal-forms} when
  \(\workingtth\) has operators with variable binding such as
  dependent function types and lambda abstraction; see
  \cref{prop:section-component-ty2} below.
\end{remark}

\subsection{Normalization results}
\label{sec:norm-results-1}

Suppose we have done \cref{task:uniqueness-of-normal-forms} and assume
\cref{axm:working-logos-2-mode-sketch,axm:working-logos-2-axioms-of-working-logos-1,axm:working-logos-2-results-of-working-logos-1,axm:working-logos-2-normalization-model,axm:working-logos-2-section,axm:working-logos-2-tiny,axm:working-logos-2-nm-comparison-1,axm:working-logos-2-nm-comparison-2,axm:working-logos-2-var-to-tm}
in type theory. By
\cref{prop:uniqueness-of-nf-var,task:uniqueness-of-normal-forms} and
by the initiality of \(\workingNfalg\), we have a section of
\(\workingUniqPred\). By definition, this shows that the propositions
\labelcref{eq:isnfty-contractible,eq:isnftm-contractible} are
inhabited with centers \(\ilNormalize_{\ilTy}(\ty)\) and
\(\ilNormalize_{\ilTm}(\tm)\), respectively.

Apply the results to the \(\infty\)-logos \(\workingLogosI\). Since
\cref{eq:isnfty-contractible,eq:isnftm-contractible} are defined in
the universe \(\ilUniv_{\modeRen \lor \modeSyn}\) which corresponds to
the full subcategory \(\workingLogos \subset \workingLogosI\), it follows
that \(\workingLogos\) also validates
\cref{eq:isnfty-contractible,eq:isnftm-contractible}.

\subsection{Running example}
\label{sec:function-types-1}

We consider the case when \(\workingtth =
\basicTT_{\FunMark}\). Before doing
\cref{task:uniqueness-of-normal-forms}, we extract more components
from the section \(\workingSectionIn\). \Cref{cst:section-type-1}
shows that the section \(\workingSectionIn\) acts on types and
terms. We can show that it also acts on families of types and
terms. This is where we need \cref{axm:working-logos-2-tiny}.

\begin{proposition}
  \label{prop:section-component-ty2}
  Assume
  \cref{axm:working-logos-2-mode-sketch,axm:working-logos-2-axioms-of-working-logos-1,axm:working-logos-2-normalization-model,axm:working-logos-2-section,axm:working-logos-2-tiny}. Then
  the type
  \begin{align*}
    \begin{autobreak}
      \opMode_{\modeSect}(\ilExists_{\tyI_{\modeNm} : \ilForall_{\tm_{\modeSyn}}\ty_{\modeNm}(\tm_{\modeSyn}) \to \workingNmModelIn(\tyI_{\modeSyn}(\tm_{\modeSyn}))}
      \ilForall_{\tm_{\modeSyn}}
      \ilForall_{\tm_{\modeNm}}
      \ty_{\modeSect}(\tm_{\modeNm})
      \to \workingSectionIn(\tyI_{\modeNm}(\tm_{\modeSyn}, \tm_{\modeNm})))
    \end{autobreak}
  \end{align*}
  is contractible for any \(\ty_{\modeSyn} : \workingTyTmStr\),
  \(\ty_{\modeNm} : \workingNmModelIn(\ty_{\modeSyn})\),
  \(\ty_{\modeSect} : \workingSectionIn(\ty_{\modeNm})\),
  and \(\tyI_{\modeSyn} : \ty_{\modeSyn} \to \workingTyTmStr\).
\end{proposition}
\begin{proof}
  We have
  \begin{EqReasoning}
    \begin{align*}
      & \term{\opMode_{\modeSect}(\ilExists_{\tyI_{\modeNm} : \ilForall_{\tm_{\modeSyn}}\ty_{\modeNm}(\tm_{\modeSyn}) \to \workingNmModelIn(\tyI_{\modeSyn}(\tm_{\modeSyn}))}\ilForall_{\tm_{\modeSyn}}\ilForall_{\tm_{\modeNm}}\ty_{\modeSect}(\tm_{\modeNm}) \to \workingSectionIn(\tyI_{\modeNm}(\tm_{\modeSyn}, \tm_{\modeNm})))} \\
      \simeq & \by{\(\workingNmModelIn(\tyI_{\modeSyn}(\tm_{\modeSyn})) \simeq (\ty_{\modeSect}(\tm_{\modeNm}) \to \workingNmModelIn(\tyI_{\modeSyn}(\tm_{\modeSyn})))\) since \(\modeSect \le {}^{\orthMark}\modeNm\)} \\
      & \term{\opMode_{\modeSect}(\ilExists_{\tyI_{\modeNm} : \ilForall_{\tm_{\modeSyn}}\ilForall_{\tm_{\modeNm} : \ty_{\modeNm}(\tm_{\modeSyn})}\ty_{\modeSect}(\tm_{\modeNm}) \to \workingNmModelIn(\tyI_{\modeSyn}(\tm_{\modeSyn}))}\ilForall_{\tm_{\modeSyn}}\ilForall_{\tm_{\modeNm}}\ilForall_{\tm_{\modeSect}}\workingSectionIn(\tyI_{\modeNm}(\tm_{\modeSyn}, \tm_{\modeNm}, \tm_{\modeSect})))} \\
      \simeq & \by{\(\ilForall\) distributes over \(\ilExists\)} \\
      & \term{\opMode_{\modeSect}(\ilForall_{\tm_{\modeSyn}}\ilForall_{\tm_{\modeNm} : \ty_{\modeNm}(\tm_{\modeSyn})}\ty_{\modeSect}(\tm_{\modeNm}) \to \ilExists_{\tyI_{\modeNm} : \workingNmModelIn(\tyI_{\modeSyn}(\tm_{\modeSyn}))}\workingSectionIn(\tyI_{\modeNm}))} \\
      \simeq & \by{\cref{axm:working-logos-2-tiny}} \\
      & \term{\ilForall_{\tm_{\modeSyn}}\ilForall_{\tm_{\modeNm} : \ty_{\modeNm}(\tm_{\modeSyn})}\ty_{\modeSect}(\tm_{\modeNm}) \to \opMode_{\modeSect}(\ilExists_{\tyI_{\modeNm} : \workingNmModelIn(\tyI_{\modeSyn}(\tm_{\modeSyn}))}\workingSectionIn(\tyI_{\modeNm}))},
    \end{align*}
  \end{EqReasoning}
  and the last type is contractible since \(\workingSectionIn\) is a
  section.
\end{proof}

By taking the center of contraction, we can see that
\(\workingSectionIn\) acts on families of types and similarly on
families of \dots of families of types/terms. Since
\(\ilSectApp_{\ilTy_{1}}^{\modeRen}(\workingSectionIn)\)
(\cref{cst:working-logos-2-section-to-nn}) is essentially the
composite of \(\ilSectApp_{\ilTy_{1}}(\workingSectionIn)\),
\(\ilCorrApp_{\ilTy_{1}}(\workingCorr)\), and
\(\ilCorrApp_{\ilTy_{1}}(\workingCorrI)\), it is also extended to a
function on families of \dots of families of types/terms. For the
record, we have functions such as
\begin{align*}
  \begin{autobreak}
    \MoveEqLeft
    \ilSectApp_{\ilTy_{2}}^{\modeRen}(\workingSectionIn) :
    \ilForall_{\implicit{\ty : \workingTyTmStr}}
    \ilForall_{\tyI : \ty \to \workingTyTmStr}
    \ilForall_{\implicit{\tm}}
    \ilSectApp_{\ilTy_{1}}^{\modeRen}(\workingSectionIn, \ty)(\tm)
    \to \workingNmDTyTmStr(\tyI(\tm))
  \end{autobreak}
  \\
  \begin{autobreak}
    \MoveEqLeft
    \ilSectApp_{\ilTm_{2}}^{\modeRen}(\workingSectionIn) :
    \ilForall_{\implicit{\ty : \workingTyTmStr}}
    \ilForall_{\implicit{\tyI : \ty \to \workingTyTmStr}}
    \ilForall_{\tmI : \ilForall_{\tm}\tyI(\tm)}
    \ilForall_{\implicit{\tm}}
    \ilForall_{\tm_{0} : \ilSectApp_{\ilTy_{1}}^{\modeRen}(\workingSectionIn, \ty)(\tm)}
    \ilSectApp_{\ilTy_{2}}^{\modeRen}(\workingSectionIn, \tyI, \tm_{0})(\tmI(\tm)).
  \end{autobreak}
\end{align*}
Moreover, these commute with structural operators such as substitution
and respect all type-theoretic structures, for example
\begin{align*}
  \begin{autobreak}
    \MoveEqLeft
    \ilSectApp_{\ilTy_{1}}^{\modeRen}(\workingSectionIn, \tyI(\tm)) =
    \ilSectApp_{\ilTy_{2}}^{\modeRen}(\workingSectionIn, \tyI)(\ilSectApp_{\ilTm_{1}}^{\modeRen}(\workingSectionIn, \tm))
  \end{autobreak}
  \\
  \begin{autobreak}
    \MoveEqLeft
    \ilSectApp_{\ilTy_{1}}^{\modeRen}(\workingSectionIn, \ilFun(\ty, \tyI)) =
    \ilFun^{\neord[1]}(\ilSectApp_{\ilTy_{1}}^{\modeRen}(\workingSectionIn, \ty), \ilSectApp_{\ilTy_{2}}^{\modeRen}(\workingSectionIn, \tyI)).
  \end{autobreak}
\end{align*}
These functions and identifications are part of an \emph{infinite}
coherent structure. Our \emph{finite} set of axioms allows us to
extract arbitrarily many components of that structure.

\paragraph{\Cref{task:uniqueness-of-normal-forms}}

\begin{warning}
  In the equational reasonings below, many transport functions are
  omitted for readability. Some expressions may be even ill-formed but
  can be fixed by inserting appropriate transport functions.
\end{warning}

\subparagraph{Case \(\ilFun_{\nftyMark}(\ty_{\nftyMark}, \tyI_{\nftyMark})\)}

Let \(\ty : \workingTyTmStr\),
\(\ty_{\nftyMark} : \ilIsNfTy_{\workingNfalg}(\ty)\),
\(\tyI : \ty \to \workingTyTmStr\), and
\(\tyI_{\nftyMark} : \ilForall_{\var}\ilIsVar(\var) \to
\ilIsNfTy_{\workingNfalg}(\tyI(\var))\), and suppose induction
hypotheses
\begin{align*}
  \begin{autobreak}
    \MoveEqLeft
    \ilNormalize_{\ilTy}(\ty) =
    \ty_{\nftyMark}
  \end{autobreak}
  \\
  \begin{autobreak}
    \MoveEqLeft
    \ilForall_{\var}
    \ilForall_{\var_{\varMark}}
    \ilNormalize_{\ilTy}(\tyI(\var)) =
    \tyI_{\nftyMark}(\var, \var_{\varMark}).
  \end{autobreak}
\end{align*}
Then we have
\begin{EqReasoning}
  \begin{align*}
    & \term{\ilNormalize_{\ilTy}(\ilFun(\ty, \tyI))} \\
    = & \by{definition} \\
    & \term{\ilnfty(\ilSectApp_{\ilTy_{1}}^{\modeRen}(\workingSectionIn, \ilFun(\ty, \tyI)))} \\
    = & \by{\(\workingSectionIn\) respects function types} \\
    & \term{\ilnfty(\ilFun^{\neord[1]}(\ilSectApp_{\ilTy_{1}}^{\modeRen}(\workingSectionIn, \ty), \ilSectApp_{\ilTy_{2}}^{\modeRen}(\workingSectionIn, \tyI)))} \\
    = & \by{definition} \\
    & \term{\ilFun_{\nftyMark}(\ilnfty(\ilSectApp_{\ilTy_{1}}^{\modeRen}(\workingSectionIn, \ty)),} \\
    & \term{\quad\ilAbs \var \var_{\varMark}. \ilnfty(\ilSectApp_{\ilTy_{2}}^{\modeRen}(\workingSectionIn, \tyI, \ilVarToTm(\blank, \var_{\varMark})))).}
  \end{align*}
\end{EqReasoning}
For the second argument of the last expression, we have
\begin{EqReasoning}
  \begin{align*}
    & \term{\ilSectApp_{\ilTy_{2}}^{\modeRen}(\workingSectionIn, \tyI, \ilVarToTm(\blank, \var_{\varMark}))} \\
    = & \by{\cref{axm:working-logos-2-var-to-tm}} \\
    & \term{\ilSectApp_{\ilTy_{2}}^{\modeRen}(\workingSectionIn, \tyI, \ilSectApp_{\ilTm_{1}}^{\modeRen}(\workingSectionIn, \var))} \\
    = & \by{\(\workingSectionIn\) commutes with substitution} \\
    & \term{\ilSectApp_{\ilTy_{1}}^{\modeRen}(\workingSectionIn, \tyI(\var)).}
  \end{align*}
\end{EqReasoning}
Then, by the induction hypotheses, we conclude that
\begin{math}
  \ilNormalize_{\ilTy}(\ilFun(\ty, \tyI)) =
  \ilFun_{\nftyMark}(\ty_{\nftyMark}, \tyI_{\nftyMark}).
\end{math}

\subparagraph{Case \(\ilabs_{\nftmMark}(\tmI_{\nftmMark})\)}

Let \(\ty : \workingTyTmStr\), \(\tyI : \ty \to \workingTyTmStr\),
\(\tmI : \ilForall_{\tm}\tyI(\tm)\), and
\(\tmI_{\nftmMark} : \ilForall_{\var}\ilIsVar(\var) \to
\ilIsNfTm_{\workingNfalg}(\tmI(\var))\), and suppose induction
hypothesis
\begin{align*}
  \begin{autobreak}
    \MoveEqLeft
    \ilForall_{\var}
    \ilForall_{\var_{\varMark}}
    \ilNormalize_{\ilTm}(\tmI(\var)) =
    \tmI_{\nftmMark}(\var, \var_{\varMark}).
  \end{autobreak}
\end{align*}
Then we have
\begin{EqReasoning}
  \begin{align*}
    & \term{\ilNormalize_{\ilTm}(\ilabs(\tmI))} \\
    = & \by{\(\workingSectionIn\) respects function types} \\
    & \term{\ilreify(\blank, \ilabs^{\neord[1]}(\ilSectApp_{\ilTm_{2}}^{\modeRen}(\workingSectionIn, \tmI)))} \\
    = & \by{definition} \\
    & \term{\ilabs_{\nftmMark}(\ilAbs \var \var_{\varMark}. \ilreify(\blank, \ilSectApp_{\ilTm_{2}}^{\modeRen}(\workingSectionIn, \tmI, \ilVarToTm(\blank, \var_{\varMark}))))} \\
    = & \by{\cref{axm:working-logos-2-var-to-tm}, and \(\workingSectionIn\) commutes with substitution} \\
    & \term{\ilabs_{\nftmMark}(\ilAbs \var \var_{\varMark}. \ilreify(\blank, \ilSectApp_{\ilTm_{1}}^{\modeRen}(\workingSectionIn, \tmI(\var))))} \\
    = & \by{induction hypothesis} \\
    & \term{\ilabs_{\nftmMark}(\tmI_{\nftmMark})}.
  \end{align*}
\end{EqReasoning}

\subparagraph{Case \(\ilapp_{\netmMark}(\map_{\netmMark}, \tm_{\nftmMark})\)}

Let \(\ty : \workingTyTmStr\), \(\tyI : \ty \to \workingTyTmStr\),
\(\map : \ilFun(\ty, \tyI)\),
\(\map_{\netmMark} : \ilIsNeTm_{\workingNfalg}(\map)\), \(\tm : \ty\),
and \(\tm_{\nftmMark} : \ilIsNfTm_{\workingNfalg}(\tm)\), and suppose
induction hypotheses
\begin{align*}
  \begin{autobreak}
    \MoveEqLeft
    \ilSectApp_{\ilTm_{1}}^{\modeRen}(\workingSectionIn, \map) =
    \ilreflect(\blank, \map_{\netmMark})
  \end{autobreak}
  \\
  \begin{autobreak}
    \MoveEqLeft
    \ilNormalize_{\ilTm}(\tm) =
    \tm_{\nftmMark}.
  \end{autobreak}
\end{align*}
Then we have
\begin{EqReasoning}
  \begin{align*}
    & \term{\ilSectApp_{\ilTm_{1}}^{\modeRen}(\workingSectionIn, \ilapp(\map, \tm))} \\
    = & \by{\(\workingSectionIn\) respects function types} \\
    & \term{\ilapp^{\neord[1]}(\ilSectApp_{\ilTm_{1}}^{\modeRen}(\workingSectionIn, \map), \ilSectApp_{\ilTm_{1}}^{\modeRen}(\workingSectionIn, \tm))} \\
    = & \by{induction hypothesis for \(\map_{\netmMark}\)} \\
    & \term{\ilapp^{\neord[1]}(\ilreflect(\blank, \map_{\netmMark}), \ilSectApp_{\ilTm_{1}}^{\modeRen}(\workingSectionIn, \tm))} \\
    = & \by{definition} \\
    & \term{\ilreflect(\blank, \ilapp_{\netmMark}(\map_{\netmMark}, \ilreify(\blank, \ilSectApp_{\ilTm_{1}}^{\modeRen}(\workingSectionIn, \tm))))} \\
    = & \by{induction hypothesis for \(\tm_{\nftmMark}\)} \\
    & \term{\ilreflect(\blank, \ilapp_{\netmMark}(\map_{\netmMark}, \tm_{\nftmMark}))}.
  \end{align*}
\end{EqReasoning}

\subparagraph{Other cases}

The other cases are straightforward by similar calculation.

\section{Coherence results}
\label{sec:coherence-theorem}

We continue to work with the initial model \(\workingInitModel\) of
\(\workingtth\) and the initial renaming model
\(\workingRenCounit : \workingRenModel \to \workingInitModel\). We show
that the \(\infty\)-logos
\(\workingLogos = \workingLogos_{\workingRenCounit}\) validates the
following propositions.
\begin{align}
  \label{eq:ty-isset}
  \begin{autobreak}
    \MoveEqLeft
    \ilIsSet(\opMode_{\modeRen} \workingTyTmStr)
  \end{autobreak}
  \\
  \label{eq:tm-isset}
  \begin{autobreak}
    \MoveEqLeft
    \ilForall_{\ty : \workingTyTmStr}
    \ilIsSet(\opMode_{\modeRen} \ty)
  \end{autobreak}
\end{align}

\Cref{eq:ty-isset,eq:tm-isset} will be proved by
calculating the identity types of the inductive type families
\(\ilIsNfTy_{\workingNfalg}\), \(\ilIsNfTm_{\workingNfalg}\), and
\(\ilIsNeTm_{\workingNfalg}\) by induction, but then we will have to
know about base cases. The following will be used in the base case of
\(\ilVarToNe\).

\begin{proposition}
  \label{prop:renaming-model-tm-discrete}
  Let \(\model\) be a model of \(\basicTT\) and take the initial
  renaming model \(\Ren_{\model} \to \model\). Then
  \(\Tm(\Ren_{\model})\) is a \(0\)-truncated object in
  \(\RFib_{\Ctx(\Ren_{\model})}\).
\end{proposition}

Intuitively, an object of \(\Ctx(\Ren_{\model})\) is a ``context''
\((\var_{1} : \ty_{1}, \dots, \var_{\nat} : \ty_{\nat})\), and the
fiber of \(\Tm(\Ren_{\model})\) over it is the set of variables
\(\{\var_{1}, \dots, \var_{\nat}\}\). Then \(\Tm(\Ren_{\model})\) is
\(0\)-truncated by construction. The actual proof of
\cref{prop:renaming-model-tm-discrete} is slightly different and given
in \cref{sec:proof-proposition}.

\subsection{Axiomatic setup}
\label{sec:axiomatic-setup-2}

Our third working \(\infty\)-logos is again \(\workingLogos\), but we
now know that \(\workingLogos\) satisfies more axioms.

\begin{axiom}
  \label{axm:working-logos-3-axioms-from-working-logos-1}
  \Cref{axm:working-logos-1-mode-sketch,axm:working-logos-1-type-term-structure,axm:working-logos-1-variable}
  are assumed.
\end{axiom}

\begin{axiom}
  \label{axm:working-logos-3-results-from-working-logos-2}
  Presuppose
  \cref{axm:working-logos-3-axioms-from-working-logos-1}. The
  propositions
  \labelcref{eq:isnfty-contractible,eq:isnftm-contractible} are
  assumed to be true.
\end{axiom}

By \cref{prop:renaming-model-tm-discrete}, we have:

\begin{axiom}
  \label{axm:working-logos-3-variable-discrete}
  Presuppose
  \cref{axm:working-logos-3-axioms-from-working-logos-1}. The type
  \begin{math}
    \opMode_{\modeRen}(\ilExists_{\ty : \workingTyTmStr}\ilExists_{\var : \ty}\ilIsVar(\var))
  \end{math}
  is assumed to be \(0\)-truncated.
\end{axiom}

Finally, since
\(\workingRenCounit : \workingRenModel \to \workingInitModel\) is a
morphism of models of \(\basicTT\), we have the following by
\cref{rem:modal-tininess}.

\begin{axiom}
  \label{axm:working-logos-3-variable-tiny}
  Presuppose
  \cref{axm:working-logos-3-axioms-from-working-logos-1}. For any
  \(\ty : \workingTyTmStr\) and
  \(\tyX : \ilForall_{\var : \ty}\ilIsVar(\var) \to \ilUniv_{\modeRen
    \lor \modeSyn}\), the canonical function
  \begin{align*}
    \begin{autobreak}
      \MoveEqLeft
      \opMode_{\modeRen}(\ilForall_{\var}\ilForall_{\var_{\varMark}}\tyX(\var, \var_{\varMark}))
      \to (\ilForall_{\var}\ilForall_{\var_{\varMark}}\opMode_{\modeRen}\tyX(\var, \var_{\varMark}))
    \end{autobreak}
  \end{align*}
  is assumed to be an equivalence.
\end{axiom}

\subsection{Coherence predicate}
\label{sec:coherence-predicate}

We work in type theory.

\begin{construction}
  Assume \cref{axm:working-logos-3-axioms-from-working-logos-1}. We
  define a displayed normal form predicate
  \(\workingCohPred\) over \(\workingNfalg\) by
  \begin{align*}
    \begin{autobreak}
      \MoveEqLeft
      \ilIsNfTy^{\neord[1]}_{\workingCohPred}(\ty_{\nftyMark}) \defeq
      \ilIsContr(\opMode_{\modeRen} \ilLoop(\ty, \ty_{\nftyMark}))
    \end{autobreak}
    \\
    \begin{autobreak}
      \MoveEqLeft
      \ilIsNfTm^{\neord[1]}_{\workingCohPred}(\tm_{\nftmMark}) \defeq
      \ilIsContr(\opMode_{\modeRen} \ilLoop(\tm, \tm_{\nftmMark}))
    \end{autobreak}
    \\
    \begin{autobreak}
      \MoveEqLeft
      \ilIsNeTm^{\neord[1]}_{\workingCohPred}(\tm_{\netmMark}) \defeq
      \ilIsContr(\opMode_{\modeRen} \ilLoop(\ty, \tm, \tm_{\netmMark}))
    \end{autobreak}
  \end{align*}
  over the context
  \begin{math}
    \ty : \workingTyTmStr,
    \ty_{\nftyMark} : \ilIsNfTy_{\workingNfalg}(\ty),
    \tm : \ty,
    \tm_{\nftyMark} : \ilIsNfTm_{\workingNfalg}(\tm),
    \tm_{\netmMark} : \ilIsNeTm_{\workingNfalg}(\tm).
  \end{math}
  Here, the loop space \(\ilLoop(\blank)\) is taken in appropriate
  dependent pair types.
\end{construction}

We want to show that \(\workingCohPred\) is a displayed
\(\workingtth\)-normal form algebra over \(\workingNfalg\) supporting
variables so that it admits a section by the initiality of
\(\workingNfalg\). This is done by calculating the identity types of
\(\ilIsNfTy_{\workingNfalg}\), \(\ilIsNfTm_{\workingNfalg}\), and
\(\ilIsNeTm_{\workingNfalg}\).

\begin{example}
  \label{exm:nfalg-path-var}
  Assume
  \cref{axm:working-logos-3-axioms-from-working-logos-1}. Regardless
  the presentation of \(\workingtth\), one can calculate the identity types
  for the constructor \(\ilVarToNe\) as follows. Let
  \(\ty : \workingTyTmStr\), \(\var : \ty\), and
  \(\var_{\varMark} : \ilIsVar(\var)\). We define
  \begin{math}
    \ilCode :
    \ilForall_{\tyI : \workingTyTmStr}
    \ilForall_{\tmI : \tyI}
    \ilIsNeTm_{\workingNfalg}(\tmI)
    \to \ilUniv_{\modeRen}
  \end{math}
  by
  \begin{math}
    \ilCode(\tyI, \varI, \ilVarToNe(\varI_{\varMark})) \defeq
    \opMode_{\modeRen}((\ty, \var, \var_{\varMark}) = (\tyI, \varI, \varI_{\varMark}))
  \end{math}
  and
  \begin{math}
    \ilCode(\blank, \blank, \blank) \defeq \opMode_{\modeRen} \ilEmpty
  \end{math}
  for all the other cases. It is routine to establish an equivalence
  \begin{equation*}
    \opMode_{\modeRen}((\ty, \var, \ilVarToNe(\var_{\varMark})) = (\tyI, \tmI, \tmI_{\neMark}))
    \simeq \ilCode(\tyI, \tmI, \tmI_{\neMark})
  \end{equation*}
  by the encode-decode method \parencite[Section 2.12]{hottbook}
  or the fundamental theorem of identity types \parencite[Theorem
  5.8.2]{hottbook}. In particular,
  \begin{align*}
    \begin{autobreak}
      \MoveEqLeft
      \opMode_{\modeRen}((\ty, \var, \ilVarToNe(\var_{\varMark})) = (\tyI, \varI, \ilVarToNe(\varI_{\varMark}))) \simeq
      \opMode_{\modeRen}((\ty, \var, \var_{\varMark}) = (\tyI, \varI, \varI_{\varMark})).
    \end{autobreak}
  \end{align*}
\end{example}

\begin{proposition}
  \label{prop:coherence-var}
  Assume
  \cref{axm:working-logos-3-axioms-from-working-logos-1,axm:working-logos-3-variable-discrete}. Then
  \(\workingCohPred\) supports variables, that is, we have a function
  \begin{align*}
    \begin{autobreak}
      \MoveEqLeft
      \ilForall_{\ty : \workingTyTmStr}
      \ilForall_{\var : \ty}
      \ilForall_{\var_{\varMark} : \ilIsVar(\var)}
      \ilIsNeTm^{\neord[1]}_{\workingCohPred}(\ilVarToNe(\var_{\varMark})).
    \end{autobreak}
  \end{align*}
\end{proposition}
\begin{proof}
  Let \(\ty : \workingTyTmStr\), \(\var : \ty\), and
  \(\var_{\varMark} : \ilIsVar(\var)\). We show that
  \begin{math}
    \opMode_{\modeRen} \ilLoop(\ty, \var, \ilVarToNe(\var_{\varMark}))
  \end{math}
  is contractible. By \cref{exm:nfalg-path-var}, we have
  \begin{math}
    \opMode_{\modeRen} \ilLoop(\ty, \var, \ilVarToNe(\var_{\varMark}))
    \simeq \opMode_{\modeRen} \ilLoop(\ty, \var, \var_{\varMark}),
  \end{math}
  but the latter is contractible by
  \cref{axm:working-logos-3-variable-discrete}.
\end{proof}

\begin{task}
  \label{task:coherence}
  Assume
  \cref{axm:working-logos-3-axioms-from-working-logos-1,axm:working-logos-3-results-from-working-logos-2,axm:working-logos-3-variable-tiny}. Show
  that \(\workingCohPred\) is a displayed \(\workingtth\)-normal form algebra
  over \(\nfalg\).
\end{task}

\begin{remark}
  \Cref{axm:working-logos-3-variable-discrete} is only meant to be the
  base case of induction. We do not need it in \cref{task:coherence}.
\end{remark}

\begin{remark}
  \Cref{axm:working-logos-3-variable-tiny} is used in
  \cref{task:coherence} when \(\workingtth\) has variable binding operators.
\end{remark}

\subsection{Coherence results}
\label{sec:coherence-results}

Suppose we have done \cref{task:coherence} and assume
\cref{axm:working-logos-3-axioms-from-working-logos-1,axm:working-logos-3-results-from-working-logos-2,axm:working-logos-3-variable-discrete,axm:working-logos-3-variable-tiny}
in type theory. By \cref{prop:coherence-var,task:coherence} and by the
initiality of \(\workingNfalg\), we have a section of
\(\workingCohPred\). By definition and by
\cref{axm:working-logos-3-results-from-working-logos-2}, this shows
that the type \(\opMode_{\modeRen} \ilLoop(\ty)\) is contractible for
any \(\ty : \workingTyTmStr\) and that the type
\(\opMode_{\modeRen} \ilLoop(\tm)\) is contractible for any
\(\tm : \ty\). Since \(\modeRen\) is lex, these imply
\cref{eq:ty-isset,eq:tm-isset}.

Apply the results to the \(\infty\)-logos \(\workingLogos\) to show that
the object
\(\workingRenCounit_{\Ctx}^{\pbMark} \Ty(\workingInitModel)\) and the
map
\(\workingRenCounit_{\Ctx}^{\pbMark} \Tm(\workingInitModel) \to
\workingRenCounit_{\Ctx}^{\pbMark} \Ty(\workingInitModel)\) in
\(\RFib_{\Ctx(\workingRenModel)}\) are \(0\)-truncated. Since
\(\workingInitModel\) and \(\workingRenModel\) have the same types,
the context part
\(\workingRenCounit_{\Ctx} : \Ctx(\workingRenModel) \to
\Ctx(\workingInitModel)\) is surjective on objects. Then the object
\(\Ty(\workingInitModel)\) and the map
\(\Tm(\workingInitModel) \to \Ty(\workingInitModel)\) in
\(\RFib_{\Ctx(\workingInitModel)}\) are also \(0\)-truncated. Finally,
using the fact that morphisms in \(\Ctx(\workingInitModel)\) are
presented by sequences of terms (precisely, the initial model
\(\workingInitModel\) is democratic \parencite[Proposition
4.11]{nguyen2022type-arxiv}), we see that \(\Ctx(\workingInitModel)\)
is a \((1, 1)\)-category. Thus, we conclude that the initial model
\(\workingInitModel\) of \(\workingtth\) is a \(1\)-model.

\subsection{Running example}
\label{sec:function-types-2}

We consider the case when \(\workingtth = \basicTT_{\FunMark}\). One
can show that the total function of each constructor for
\(\workingNfalg\) is a \emph{\(\modeRen\)-embedding}. For example, the
function
\begin{align*}
  \begin{autobreak}
    \MoveEqLeft
    \opMode_{\modeRen}(\ilExists_{\ty}
    \ilIsNfTy_{\workingNfalg}(\ty)
    \times \ilExists_{\tyI : \ty \to \workingTyTmStr}
    \ilForall_{\var}\ilIsVar(\var) \to \ilIsNfTy_{\workingNfalg}(\tyI(\var)))
    \to \opMode_{\modeRen}(\ilExists_{\tyII}\ilIsNfTy_{\workingNfalg}(\tyII))
  \end{autobreak}
\end{align*}
induced by the constructor \(\ilFun_{\nftyMark}\) is an
embedding. Using
\cref{axm:working-logos-3-results-from-working-logos-2}, the total
function of \(\ilFun\) is also a \(\modeRen\)-embedding:
\begin{math}
  \opMode_{\modeRen}(\ilExists_{\ty : \workingTyTmStr}\ty \to \workingTyTmStr)
  \hookrightarrow \opMode_{\modeRen} \workingTyTmStr.
\end{math}
Then, for any \(\ty : \workingTyTmStr\) and
\(\tyI : \ty \to \workingTyTmStr\), the function
\begin{align*}
  \begin{autobreak}
    \MoveEqLeft
    \opMode_{\modeRen}(\ilExists_{\tmI : \ilForall_{\tm}\tyI(\tm)}\ilIsNfTm_{\workingNfalg}(\tmI))
    \to \opMode_{\modeRen}(\ilExists_{\map : \ilFun(\ty, \tyI)}\ilIsNfTm_{\workingNfalg}(\map))
  \end{autobreak}
\end{align*}
induced by the constructor \(\ilabs_{\nftmMark}\) is also an
embedding.

\paragraph{\Cref{task:coherence}}

\subparagraph{Case \(\ilFun_{\nftyMark}(\ty_{\nftyMark}, \tyI_{\nftyMark})\)}

Let \(\ty : \workingTyTmStr\),
\(\ty_{\nftyMark} : \ilIsNfTy_{\workingNfalg}(\ty)\),
\(\tyI : \ty \to \workingTyTmStr\), and
\(\tyI_{\nftyMark} : \ilForall_{\var}\ilIsVar(\var) \to
\ilIsNfTy_{\workingNfalg}(\ty)\), and suppose induction hypotheses
\begin{align*}
  \begin{autobreak}
    \MoveEqLeft
    \ilIsContr(\opMode_{\modeRen} \ilLoop(\ty, \ty_{\nftyMark}))
  \end{autobreak}
  \\
  \begin{autobreak}
    \MoveEqLeft
    \ilForall_{\var}
    \ilForall_{\var_{\varMark}}
    \ilIsContr(\opMode_{\modeRen} \ilLoop(\tyI(\var), \tyI_{\nftyMark}(\var, \var_{\varMark}))).
  \end{autobreak}
\end{align*}
We have
\begin{EqReasoning}
  \begin{align*}
    & \term{\opMode_{\modeRen} \ilLoop(\ilFun(\ty, \tyI), \ilFun_{\nftyMark}(\ty_{\nftyMark}, \tyI_{\nftyMark}))} \\
    \simeq & \by{\(\ilFun_{\nftyMark}\) is a \(\modeRen\)-embedding} \\
    & \term{\opMode_{\modeRen} \ilLoop(\ty, \ty_{\nftyMark}, \tyI, \tyI_{\nftyMark})} \\
    \simeq & \by{induction hypothesis for \(\ty_{\nftyMark}\)} \\
    & \term{\opMode_{\modeRen} \ilLoop(\tyI, \tyI_{\nftyMark})} \\
    \simeq & \by{\(\workingTyTmStr \simeq (\ilIsVar(\var) \to \workingTyTmStr)\) as \(\modeRen \le {}^{\orthMark}\modeSyn\)} \\
    & \term{\opMode_{\modeRen} \ilLoop(\ilAbs \var \var_{\varMark}. \tyI(\var), \tyI_{\nftyMark})} \\
    \simeq & \by{function extensionality} \\
    & \term{\opMode_{\modeRen}(\ilForall_{\var}\ilForall_{\var_{\varMark}}\ilLoop(\tyI(\var), \tyI_{\nftyMark}(\var, \var_{\varMark})))} \\
    \simeq & \by{\cref{axm:working-logos-3-variable-tiny}} \\
    & \term{\ilForall_{\var}\ilForall_{\var_{\varMark}}\opMode_{\modeRen} \ilLoop(\tyI(\var), \tyI_{\nftyMark}(\var, \var_{\varMark}))},
  \end{align*}
\end{EqReasoning}
and the last type is contractible by the induction hypothesis for
\(\tyI_{\nftyMark}\).

\subparagraph{Case \(\ilabs_{\nftmMark}(\tmI_{\nftmMark})\)}

Let \(\ty : \workingTyTmStr\), \(\tyI : \ty \to \workingTyTmStr\),
\(\tmI : \ilForall_{\tm}\tyI(\tm)\), and
\(\tmI_{\nftmMark} : \ilForall_{\var}\ilIsVar(\var) \to
\ilIsNfTm_{\workingNfalg}(\tmI(\var))\), and suppose induction
hypothesis
\begin{align*}
  \begin{autobreak}
    \MoveEqLeft
    \ilForall_{\var}
    \ilForall_{\var_{\varMark}}
    \ilIsContr(\opMode_{\modeRen} \ilLoop(\tmI(\var), \tmI_{\nftmMark}(\var, \var_{\varMark}))).
  \end{autobreak}
\end{align*}
We have
\begin{EqReasoning}
  \begin{align*}
    & \term{\opMode_{\modeRen} \ilLoop(\ilabs(\tmI), \ilabs_{\nftmMark}(\tmI_{\nftmMark}))} \\
    \simeq & \by{\(\ilabs_{\nftmMark}\) is a \(\modeRen\)-embedding} \\
    & \term{\opMode_{\modeRen} \ilLoop(\tmI, \tmI_{\nftmMark})} \\
    \simeq & \by{\cref{axm:working-logos-3-variable-tiny} and \(\modeRen \le {}^{\orthMark}\modeSyn\)} \\
    & \term{\ilForall_{\var}\ilForall_{\var_{\varMark}}\opMode_{\modeRen} \ilLoop(\tmI(\var), \tmI_{\nftmMark}(\var, \var_{\varMark}))},
  \end{align*}
\end{EqReasoning}
and the last type is contractible by the induction hypothesis.

\subparagraph{Case \(\ilapp_{\netmMark}(\map_{\netmMark}, \tm_{\nftmMark})\)}

Let \(\ty : \workingTyTmStr\), \(\tyI : \ty \to \workingTyTmStr\),
\(\map : \ilFun(\ty, \tyI)\),
\(\map_{\netmMark} : \ilIsNeTm_{\workingNfalg}(\map)\), \(\tm : \ty\),
and \(\tm_{\nftmMark} : \ilIsNfTm_{\workingNfalg}(\tm)\), and suppose
induction hypotheses
\begin{align*}
  \begin{autobreak}
    \MoveEqLeft
    \ilIsContr(\opMode_{\modeRen} \ilLoop(\ilFun(\ty, \tyI), \map, \map_{\netmMark}))
  \end{autobreak}
  \\
  \begin{autobreak}
    \MoveEqLeft
    \ilIsContr(\opMode_{\modeRen} \ilLoop(\tm, \tm_{\nftmMark})).
  \end{autobreak}
\end{align*}
We have
\begin{EqReasoning}
  \begin{align*}
    & \term{\opMode_{\modeRen} \ilLoop(\tyI(\tm), \ilapp(\map, \tm), \ilapp_{\netmMark}(\map_{\netmMark}, \tm_{\nftmMark}))} \\
    \simeq & \by{\(\ilapp_{\netmMark}\) is a \(\modeRen\)-embedding} \\
    & \term{\opMode_{\modeRen} \ilLoop(\ty, \tyI, \map, \map_{\netmMark}, \tm, \tm_{\nftmMark})} \\
    \simeq & \by{\(\ilFun\) is a \(\modeRen\)-embedding} \\
    & \term{\opMode_{\modeRen} \ilLoop(\ilFun(\ty, \tyI), \map, \map_{\netmMark}, \tm, \tm_{\nftmMark})},
  \end{align*}
\end{EqReasoning}
and the last type is contractible by the induction hypotheses.

\subparagraph{Cases \(\ilNat_{\nftyMark}\), \(\ilzero_{\nftmMark}\), and \(\ilsucc_{\nftmMark}(\nat_{\nftmMark})\)}

These cases require no new idea.

\subparagraph{Case \(\ilNeToNf_{\ilNat}(\nat_{\netmMark})\)}

Let \(\nat : \ilNat\) and
\(\nat_{\netmMark} : \ilIsNeTm_{\workingNfalg}(\nat)\), and suppose
induction hypothesis
\begin{align*}
  \begin{autobreak}
    \MoveEqLeft
    \ilIsContr(\opMode_{\modeRen} \ilLoop(\ilNat, \nat, \nat_{\netmMark})).
  \end{autobreak}
\end{align*}
We have
\begin{EqReasoning}
  \begin{align*}
    & \term{\opMode_{\modeRen} \ilLoop(\nat, \ilNeToNf_{\ilNat}(\nat_{\netmMark}))} \\
    \simeq & \by{\(\ilNat\) is a \(\modeRen\)-embedding} \\
    & \term{\opMode_{\modeRen} \ilLoop(\ilNat, \nat, \ilNeToNf_{\ilNat}(\nat_{\netmMark}))} \\
    \simeq & \by{\(\ilNeToNf_{\ilNat}\) is a \(\modeRen\)-embedding} \\
    & \term{\opMode_{\modeRen} \ilLoop(\nat, \nat_{\netmMark})} \\
    \simeq & \by{\(\ilNat\) is a \(\modeRen\)-embedding} \\
    & \term{\opMode_{\modeRen} \ilLoop(\ilNat, \nat, \nat_{\netmMark})},
  \end{align*}
\end{EqReasoning}
and the last type is contractible by the induction hypothesis.

\subparagraph{Case \(\ilind_{\ilNat, \netmMark}(\nat_{\netmMark}, \ty_{\nftmMark}, \tm_{\zeroMark, \nftmMark}, \tm_{\succMark, \nftmMark})\)}

Let \(\nat : \ilNat\),
\(\nat_{\netmMark} : \ilIsNeTm_{\workingNfalg}(\nat)\),
\(\ty : \ilNat \to \workingTyTmStr\),
\(\ty_{\nftyMark} : \ilForall_{\var}\ilIsVar(\var) \to
\ilIsNfTy_{\workingNfalg}(\ty(\var))\),
\(\tm_{\zeroMark} : \ty(\ilzero)\),
\(\tm_{\zeroMark, \nftmMark} :
\ilIsNfTm_{\workingNfalg}(\tm_{\zeroMark})\),
\(\tm_{\succMark} : \ilForall_{\var}\ty(\var) \to \ty(\ilsucc(\var))\),
and
\(\tm_{\succMark, \netmMark} : \ilForall_{\var}\ilIsVar(\var) \to
\ilForall_{\varI}\ilIsVar(\varI) \to
\ilIsNfTm_{\workingNfalg}(\tm_{\succMark}(\var, \varI))\), and suppose
induction hypotheses
\begin{align*}
  \begin{autobreak}
    \MoveEqLeft
    \ilIsContr(\opMode_{\modeRen} \ilLoop(\ilNat, \nat, \nat_{\netmMark}))
  \end{autobreak}
  \\
  \begin{autobreak}
    \MoveEqLeft
    \ilForall_{\var}
    \ilForall_{\var_{\varMark}}
    \ilIsContr(\opMode_{\modeRen} \ilLoop(\ty(\var), \ty_{\nftmMark}(\var, \var_{\varMark})))
  \end{autobreak}
  \\
  \begin{autobreak}
    \MoveEqLeft
    \ilIsContr(\opMode_{\modeRen} \ilLoop(\tm_{\zeroMark}, \tm_{\zeroMark, \nftmMark}))
  \end{autobreak}
  \\
  \begin{autobreak}
    \MoveEqLeft
    \ilForall_{\var}
    \ilForall_{\var_{\varMark}}
    \ilForall_{\varI}
    \ilForall_{\varI_{\varMark}}
    \ilIsContr(\opMode_{\modeRen} \ilLoop(\tm_{\succMark}(\var, \varI), \tm_{\succMark, \nftmMark}(\var, \var_{\varMark}, \varI, \varI_{\varMark}))).
  \end{autobreak}
\end{align*}
We have
\begin{EqReasoning}
  \begin{align*}
    & \term{\opMode_{\modeRen} \ilLoop(\ty(\nat), \ilind_{\ilNat}(\nat, \ty, \tm_{\zeroMark}, \tm_{\succMark}), \ilind_{\ilNat, \netmMark}(\nat_{\netmMark}, \ty_{\nftyMark}, \tm_{\zeroMark, \nftmMark}, \tm_{\succMark, \nftmMark}))} \\
    \simeq & \by{\(\ilind_{\ilNat, \netmMark}\) is a \(\modeRen\)-embedding} \\
    & \term{\opMode_{\modeRen} \ilLoop(\nat, \nat_{\netmMark}, \ty, \ty_{\nftyMark}, \tm_{\zeroMark}, \tm_{\zeroMark, \nftmMark}, \tm_{\succMark}, \tm_{\succMark, \nftmMark})} \\
    \simeq & \by{\(\ilNat\) is a \(\modeRen\)-embedding} \\
    & \term{\opMode_{\modeRen} \ilLoop(\ilNat, \nat, \nat_{\netmMark}, \ty, \ty_{\nftyMark}, \tm_{\zeroMark}, \tm_{\zeroMark, \nftmMark}, \tm_{\succMark}, \tm_{\succMark, \nftmMark})},
  \end{align*}
\end{EqReasoning}
and the last type is proved to be contractible from the induction
hypotheses, using \cref{axm:working-logos-3-variable-tiny} to handle
variables.

\subparagraph{Cases \(\ilConstTy_{\nftyMark}\), \(\ilNeToNf_{\ilConstTy}(\tmII_{\netmMark})\), \(\ilConstTyI_{\nftyMark}(\tmII_{\nftmMark})\), and \(\ilConstTm_{\netmMark}(\tmII_{\nftmMark})\)}

Straightforward.

\subparagraph{Case \(\ilNeToNf_{\ilConstTyI}(\tmII_{\nftmMark}, \tmIII_{\netmMark})\)}

Let \(\tmII : \ilConstTy\),
\(\tmII_{\nftmMark} : \ilIsNfTm_{\workingNfalg}(\tmII)\),
\(\tmIII : \ilConstTyI(\tmII)\), and
\(\tmIII_{\netmMark} : \ilIsNeTm_{\workingNfalg}(\tmIII)\), and
suppose induction hypotheses
\begin{align*}
  \begin{autobreak}
    \MoveEqLeft
    \ilIsContr(\opMode_{\modeRen} \ilLoop(\tmII, \tmII_{\nftmMark}))
  \end{autobreak}
  \\
  \begin{autobreak}
    \MoveEqLeft
    \ilIsContr(\opMode_{\modeRen} \ilLoop(\ilConstTyI(\tmII), \tmIII, \tmIII_{\netmMark})).
  \end{autobreak}
\end{align*}
By the case of \(\ilConstTyI_{\nftyMark}(\tmII_{\nftmMark})\), the
type
\(\opMode_{\modeRen} \ilLoop(\ilConstTyI(\tmII),
\ilConstTyI_{\nftyMark}(\tmII_{\nftmMark}))\) is contractible. By
\cref{axm:working-logos-3-results-from-working-logos-2},
\(\opMode_{\modeRen} \ilLoop(\ilConstTyI(\tmII))\) is also
contractible. Then
\(\opMode_{\modeRen} \ilLoop(\tmIII, \tmIII_{\netmMark})\) is
contractible by the induction hypothesis for
\(\tmIII_{\netmMark}\). Then,
\begin{EqReasoning}
  \begin{align*}
    & \term{\opMode_{\modeRen} \ilLoop(\tmIII, \ilNeToNf_{\ilConstTyI}(\tmII_{\nftmMark}, \tmIII_{\netmMark}))} \\
    \simeq & \by{\(\opMode_{\modeRen}(\ilLoop(\ilConstTyI(\tmII)))\) is contractible} \\
    & \term{\opMode_{\modeRen} \ilLoop(\ilConstTyI(\tmII), \tmIII, \ilNeToNf_{\ilConstTyI}(\tmII_{\nftmMark}, \tmIII_{\netmMark}))} \\
    \simeq & \by{\(\ilNeToNf_{\ilConstTyI}\) is a \(\modeRen\)-embedding} \\
    & \term{\opMode_{\modeRen} \ilLoop(\tmII, \tmII_{\nftmMark}, \tmIII, \tmIII_{\netmMark}),}
  \end{align*}
\end{EqReasoning}
and the last type is contractible since
\(\opMode_{\modeRen} \ilLoop(\tmII, \tmII_{\nftmMark})\) and
\(\opMode_{\modeRen} \ilLoop(\tmIII, \tmIII_{\netmMark})\) are.

\section*{Acknowledgements}
\label{sec:acknowledgements}

The author is grateful to Peter LeFanu Lumsdaine, Jonathan Sterling,
Daniel Gratzer, Christian Sattler, and Marcelo Fiore for useful
discussions. The author is supported by KAW Grant ``Type Theory for
Mathematics and Computer Science''.

\printbibliography

\appendix

\section{Proof of Proposition \ref{prop:relative-induction-trans}}
\label{sec:proof-relative-induction}

We first give a useful construction of representable maps of right
fibrations.

\begin{construction}
  \label{lem:rep-map-equiv-comprehension}
  Let \(\cat\) be an \((\infty, 1)\)-category and \(\sh\) a right fibration over
  \(\cat\). Let \(\Rep_{\sh}\) denote the full subcategory of
  \((\RFib_{\cat})_{\slice \sh}\) spanned by the representable maps
  with codomain \(\sh\). We have a functor
  \begin{equation*}
    (\RFib_{\cat})_{\slice \sh} \to
    \Fun_{\RFib_{\cat}}((\RFib_{\cat})_{\slice \sh}, (\RFib_{\cat})^{\to})
  \end{equation*}
  that sends \(\shI \to \sh\) to the functor mapping
  \(\map : \sh' \to \sh\) to base change
  \(\map^{\pbMark} \shI \to \sh'\). By construction, it factors through
  the full subcategory
  \(\Cart_{\RFib_{\cat}}((\RFib_{\cat})_{\slice \sh}\comma
  (\RFib_{\cat})^{\to})\) consisting of those functors sending all
  morphisms in \((\RFib_{\cat})_{\slice \sh}\) to pullback squares. By
  restricting along the Yoneda embedding, we have a functor
  \begin{equation*}
    (\RFib_{\cat})_{\slice \sh} \to
    \Cart_{\cat}(\sh, \yoneda^{\pbMark} (\RFib_{\cat})^{\to}).
  \end{equation*}
  This functor is an equivalence by Yoneda. By the definition of
  representable maps, the restriction to \(\Rep_{\sh}\) induces an
  equivalence
  \begin{equation*}
    \Rep_{\sh} \simeq \Cart_{\cat}(\sh, \cat^{\to}).
  \end{equation*}
  For a representable map \(\map : \shI \to \sh\), the corresponding
  functor \(\comprProj_{\map} : \sh \to \cat^{\to}\) sends a section
  \(\el : \yoneda \ctx \to \sh\) to the morphism
  \(\comprProj_{\map}(\tm) : \compr{\el}_{\map} \to \ctx\).
\end{construction}

\begin{remark}
  The latter structure
  \(\comprProj \in \Cart_{\cat}(\sh, \cat^{\to})\) is an
  \(\infty\)-analog of the structure of \emph{comprehension category}
  \parencite{jacobs1993comprehension} and \emph{category with
    attributes} \parencite{cartmell1978generalised}.
\end{remark}

\begin{lemma}
  \label{lem:rep-map-equiv-comprehension-morphism-part}
  Let \(\model\) and \(\modelI\) be models of \(\basicTT\), let
  \(\fun_{\Ctx} : \Ctx(\model) \to \Ctx(\modelI)\) be a functor
  preserving terminal objects, and let
  \(\fun_{\Ty} : \Ty(\model) \to \Ty(\modelI)\) be a functor over
  \(\fun_{\Ctx}\). Then the space of extensions of
  \((\fun_{\Ctx}, \fun_{\Ty})\) to a morphism of models of
  \(\basicTT\) is equivalent to the space of homotopies
  \begin{equation*}
    \begin{tikzcd}
      \Ty(\model)
      \arrow[r, "\fun_{\Ty}"]
      \arrow[d, "\comprProj_{\typeof(\model)}"']
      \arrow[dr, phantom, "\simeq"{description}] &
      \Ty(\modelI)
      \arrow[d, "\comprProj_{\typeof(\modelI)}"] \\
      \Ctx(\model)^{\to}
      \arrow[r, "\fun_{\Ctx}^{\to}"'] &
      \Ctx(\modelI)^{\to}
    \end{tikzcd}
  \end{equation*}
  over \(\Ctx(\modelI)\).
\end{lemma}
\begin{proof}
  Apply \cref{lem:rep-map-equiv-comprehension} to the \(\infty\)-logos
  \(\Fun(\neord[1], \RFib)_{\fun_{\Ctx}}\) instead of
  \(\RFib_{\cat}\).
\end{proof}

\Cref{prop:relative-induction-trans} is proved by constructing the
following \emph{inserter model}.

\begin{construction}
  \label{cst:inserter-model}
  Let \(\model_{0}\) be a model of \(\basicTT\), and let
  \(\funProj : \model \to \model_{0}\) and
  \(\funProjI : \modelI \to \model_{0}\) be morphisms of models of
  \(\basicTT\). Let \(\funI : \model \to \modelI\) be a morphism over
  \(\model_{0}\) and \(\fun : \Ctx_{\model} \to \Ctx_{\modelI}\) a
  functor over \(\Ctx_{\model_{0}}\). Suppose that we have a map
  \(\vartotm : \fun_{\Ctx}^{\pbMark} \Ty(\modelI)
  \times_{\funProj_{\Ctx}^{\pbMark} \Ty(\model_{0})} \Tm(\model) \to
  \fun_{\Ctx}^{\pbMark} \Tm(\modelI)\) in \(\RFib_{\Ctx(\model)}\)
  over \(\fun_{\Ctx}^{\pbMark} \Ty(\modelI)\). We construct a model
  \(\Ins_{\model_{0}}(\fun, \funI, \vartotm)\) of \(\basicTT\).

  The base \((\infty, 1)\)-category
  \(\Ctx(\Ins_{\model_{0}}(\fun, \funI, \vartotm))\) is the inserter
  from \(\fun_{\Ctx}\) to \(\funI_{\Ctx}\) over \(\model_{0}\), that
  is, the \((\infty, 1)\)-category of pairs \((\ctx, \map)\) consisting of an
  object \(\ctx \in \Ctx(\model)\) and a morphism
  \(\map : \fun_{\Ctx}(\ctx) \to \funI_{\Ctx}(\ctx)\) over
  \(\id_{\funProj_{\Ctx}(\ctx)}\). We have the forgetful functor
  \(\insproj_{\Ctx} : \Ctx(\Ins_{\model_{0}}(\fun, \funI, \vartotm)) \to
  \Ctx(\model)\). Note that \(\insproj_{\Ctx}\) creates all limits
  preserved by \(\funI_{\Ctx}\). In particular, it creates terminal
  objects. Let
  \(\Ty(\Ins_{\model_{0}}(\fun, \funI, \vartotm)) =
  \insproj_{\Ctx}^{\pbMark} \Ty(\model)\). We construct a lift
  \begin{equation*}
    \begin{tikzcd}
      \Ty(\Ins_{\model_{0}}(\fun, \funI, \vartotm))
      \arrow[r, "\insproj_{\Ty}"]
      \arrow[d, dotted, "\comprProj_{\Ins_{\model_{0}}(\fun, \funI, \vartotm)}"'] &
      \Ty(\model)
      \arrow[d, "\comprProj_{\model}"] \\
      \Ctx(\Ins_{\model_{0}}(\fun, \funI, \vartotm))^{\to}
      \arrow[r, "\insproj_{\Ctx}^{\to}"'] &
      \Ctx(\model)^{\to}
    \end{tikzcd}
  \end{equation*}
  over \(\Ctx(\model)\). By limit creation,
  \(\comprProj_{\Ins_{\model_{0}}(\fun, \funI, \vartotm)}\)
  automatically sends all morphisms to pullback squares. Then, by
  \cref{lem:rep-map-equiv-comprehension},
  \(\Ctx(\Ins_{\model_{0}}(\fun, \funI, \vartotm))\) and
  \(\Ty(\Ins_{\model_{0}}(\fun, \funI, \var))\) are part of a model
  \(\Ins_{\model_{0}}(\fun, \funI, \vartotm)\)of \(\basicTT\), and, by
  \cref{lem:rep-map-equiv-comprehension-morphism-part},
  \(\insproj_{\Ctx}\) and \(\insproj_{\Ty}\) are part of a morphism
  \(\insproj : \Ins_{\model_{0}}(\fun, \funI, \vartotm) \to \model\).

  Assume that an object in
  \(\Ty(\Ins_{\model_{0}}(\fun, \funI, \vartotm))\) is given. By
  definition, it consists of an object \(\ctx \in \Ctx(\model)\), a
  morphism \(\map : \fun_{\Ctx}(\ctx) \to \funI_{\Ctx}(\ctx)\) over
  \(\funProj_{\Ctx}(\ctx)\), and a section
  \(\ty : \yoneda \ctx \to \Ty(\model)\). We have to find a lift
  \begin{equation*}
    \begin{tikzcd}
      \fun_{\Ctx}(\compr{\ty})
      \arrow[r, dotted]
      \arrow[d, "\fun(\comprProj(\ty))"'] &
      \funI_{\Ctx}(\compr{\ty})
      \arrow[d, "\funI(\comprProj(\ty))"] \\
      \fun_{\Ctx}(\ctx)
      \arrow[r, "\map"'] &
      \funI_{\Ctx}(\ctx).
    \end{tikzcd}
  \end{equation*}
  over \(\funProj_{\Ctx}(\compr{\ty})\). Since \(\funI\) is a morphism
  of models of \(\basicTT\), we have
  \(\yoneda (\funI_{\Ctx}(\compr{\ty})) \simeq \yoneda (\funI_{\Ctx}(\ctx))
  \times_{\Ty(\modelI)} \Tm(\modelI)\). Thus, it suffices to find a lift
  \begin{equation*}
    \begin{tikzcd}
      \yoneda (\fun_{\Ctx}(\compr{\ty}))
      \arrow[rr, dotted]
      \arrow[d, "\yoneda (\fun(\comprProj(\ty)))"'] & &
      [2ex]
      \Tm(\modelI)
      \arrow[d, "\typeof(\modelI)"] \\
      \yoneda (\fun_{\Ctx}(\ctx))
      \arrow[r, "\yoneda \map"'] &
      \yoneda (\funI_{\Ctx}(\ctx))
      \arrow[r, "\funI_{\Ty}(\ty)"'] &
      \Ty(\modelI).
    \end{tikzcd}
  \end{equation*}
  By the adjunction
  \((\fun_{\Ctx})_{\poMark} \adj \fun_{\Ctx}^{\pbMark}\), such a lift
  corresponds to a lift
  \begin{equation*}
    \begin{tikzcd}
      \yoneda \compr{\ty}
      \arrow[r, dotted]
      \arrow[d, "\yoneda (\comprProj(\ty))"'] &
      \fun_{\Ctx}^{\pbMark} \Tm(\modelI)
      \arrow[d, "\fun_{\Ctx}^{\pbMark} \typeof(\modelI)"] \\
      \yoneda \ctx
      \arrow[r, "\ty'"'] &
      \fun_{\Ctx}^{\pbMark} \Ty(\modelI),
    \end{tikzcd}
  \end{equation*}
  where \(\ty'\) is the map corresponding to
  \(\funI_{\Ty}(\ty) \comp \yoneda \map\). The lift is then given by
  \begin{align*}
    \begin{autobreak}
      \yoneda \compr{\ty}
      \xrightarrow{(\ty' \comp \yoneda (\comprProj(\ty)), \comprProjI(\ty))} \fun^{\pbMark} \Ty(\modelI) \times_{\funProj_{\Ctx}^{\pbMark} \Ty(\model_{0})} \Tm(\model)
      \xrightarrow{\vartotm} \fun_{\Ctx}^{\pbMark} \Tm(\modelI).
    \end{autobreak}
  \end{align*}
  Notice that all the steps are functorial, so we obtain a functor
  \begin{math}
    \comprProj_{\Ins_{\model_{0}}(\fun, \funI, \vartotm)} :
    \Ty(\Ins_{\model_{0}}(\fun, \funI, \var))
    \to \Ctx(\Ins_{\model_{0}}(\fun, \funI, \var))^{\to}.
  \end{math}
\end{construction}

\begin{proof}[Proof of \cref{prop:relative-induction-trans}]
  By construction, the forgetful morphism
  \(\Ins_{\workingInitModel}(\workingRenIncNm, \workingSection \comp
  \workingRenCounit, \workingVarToTm) \to \workingRenModel\) is a
  renaming model over \(\workingRenModel\) and thus admits a section
  by the initiality of \(\workingRenModel\). The \(\Ctx\)-component of
  the section corresponds to a natural transformation
  \(\workingTrans : \workingRenIncNm \To \workingSection_{\Ctx} \comp
  \workingRenCounit_{\Ctx}\). The required equivalence holds by
  construction.
\end{proof}

\section{Proof of Proposition \ref{prop:renaming-model-tm-discrete}}
\label{sec:proof-proposition}

The idea of the proof of \cref{prop:renaming-model-tm-discrete} is to
construct a renaming model \(\modelI \to \Ren_{\model}\) such that
\(\Tm(\modelI)\) is \(0\)-truncated. A retract argument then shows
that \(\Tm(\Ren_{\model})\) is also \(0\)-truncated.

\begin{construction}
  \label{cst:renaming-model-alt}
  Let \(\model\) be a model of \(\basicTT\). We construct a renaming
  model \(\RenCounitAlt : \RenAlt_{\model} \to \model\) as follows.

  We define \(\Ctx(\RenAlt_{\model})_{0}\) to be the space of
  finite sequences \((\ty_{1}, \dots, \ty_{\nat})\) of maps
  \(\ty_{\idx} : \yoneda \compr{\ty_{\idx - 1}} \to \Ty(\model)\) of
  right fibrations over \(\Ctx(\model)\), where
  \(\compr{\ty_{0}} = \objFinal\). We define a map
  \(\RenCounitAlt_{\Ctx} : \Ctx(\RenAlt_{\model})_{0} \to \Ctx(\model)\)
  by
  \begin{math}
    \RenCounitAlt_{\Ctx}(\ty_{1}, \dots, \ty_{\nat}) =
    \compr{\ty_{\nat}}.
  \end{math}
  For an element
  \(\ctx = (\ty_{1}, \dots, \ty_{\nat}) \in
  \Ctx(\RenAlt_{\model})_{0}\), we define a set
  \(\Tm(\RenAlt_{\model})_{\ctx} = \{1, \dots, \nat\}\). We also
  define a map
  \(\RenCounitAlt_{\Tm, \ctx} : \Tm(\RenAlt_{\model})_{\ctx} \to
  \Tm(\model)_{\RenCounitAlt_{\Ctx}(\ctx)}\) by sending
  \(\idx \in \Tm(\RenAlt_{\model})_{\ctx}\) to the composite of the
  projections
  \(\comprProj^{\nat}_{\idx}(\ty_{\nat}) : \yoneda \compr{\ty_{\nat}}
  \xrightarrow{\comprProj(\ty_{\nat})} \dots
  \xrightarrow{\comprProj(\ty_{\idx + 1})} \yoneda
  \compr{\ty_{\idx}}\) and
  \(\comprProjI(\ty_{\idx}) : \yoneda \compr{\ty_{\idx}} \to
  \Tm(\model)\). Let \(\catTmp\) be the \((\infty, 1)\)-category defined by the
  pullback
  \begin{equation*}
    \begin{tikzcd}
      \catTmp^{\opMark}
      \arrow[r, dotted]
      \arrow[d, dotted]
      \arrow[dr, pbMark] &
      [4ex]
      \Space^{\to}
      \arrow[d, "\cod"] \\
      \Ctx(\model)^{\opMark}
      \arrow[r, "\Tm(\model)"'] &
      \Space.
    \end{tikzcd}
  \end{equation*}
  The assignment \(\ctx \mapsto \RenCounitAlt_{\Tm, \ctx}\) determines a map
  \(\RenCounitAlt'_{\Tm} : \Ctx(\RenAlt_{\model})_{0} \to \catTmp\)
  whose projection to \(\Ctx(\model)\) is
  \(\RenCounitAlt_{\Ctx}\). Factor \(\RenCounitAlt'_{\Tm}\) as a
  surjective-on-objects functor
  \(\Ctx(\RenAlt_{\model})_{0} \to \Ctx(\RenAlt_{\model})\) followed by
  a fully faithful functor \(\Ctx(\RenAlt_{\model}) \to \catTmp\). The
  composite \(\Ctx(\RenAlt_{\model}) \to \catTmp \to \Ctx(\model)\)
  extends \(\RenCounitAlt_{\Ctx}\). The composite
  \begin{math}
    \Ctx(\RenAlt_{\model})^{\opMark}
    \to \catTmp^{\opMark}
    \to \Space^{\to}
    \xrightarrow{\dom} \Space
  \end{math}
  extends \(\Tm(\RenAlt_{\model})\) and determines a right fibration
  over \(\Ctx(\RenAlt_{\model})\). Then \(\RenCounitAlt_{\Tm, \ctx}\)
  is extended to a map of right fibrations
  \(\RenCounitAlt_{\Tm} : \Tm(\RenAlt_{\model}) \to \Tm(\model)\) over
  \(\RenCounitAlt_{\Ctx}\). Let
  \(\Ty(\RenAlt_{\model}) = \RenCounitAlt_{\Ctx}^{\pbMark}
  \Ty(\model)\) and let
  \(\RenCounitAlt_{\Ty} : \Ty(\RenAlt_{\model}) \to \Ty(\model)\) be the
  projection. \(\RenCounitAlt_{\Tm}\) induces a map
  \(\typeof(\RenAlt_{\model}) : \Tm(\RenAlt_{\model}) \to
  \Ty(\RenAlt_{\model})\) of right fibrations over
  \(\Ctx(\RenAlt_{\model})\).

  Let us describe \(\Ctx(\RenAlt_{\model})\) concretely. An object of
  \(\Ctx(\RenAlt_{\model})\) is a finite sequence
  \((\ty_{1}, \dots, \ty_{\nat})\) of maps
  \(\ty_{\idx} : \yoneda \compr{\ty_{\idx - 1}} \to \Ty(\model)\). The
  mapping space from \(\ctx\) to \(\ctxI\) is the space of triples
  \((\map, \idxmor, \pth)\) consisting of a morphism
  \(\map : \RenCounitAlt_{\Ctx}(\ctx) \to \RenCounitAlt_{\Ctx}(\ctxI)\)
  in \(\Ctx(\model)\), a map
  \(\idxmor : \Tm(\RenAlt_{\model})_{\ctxI} \to
  \Tm(\RenAlt_{\model})_{\ctx}\), and a homotopy \(\pth\) filling the
  following square.
  \begin{equation*}
    \begin{tikzcd}
      \Tm(\RenAlt_{\model})_{\ctxI}
      \arrow[r, "\idxmor"]
      \arrow[d, "\RenCounitAlt_{\Tm, \ctxI}"'] &
      \Tm(\RenAlt_{\model})_{\ctx}
      \arrow[d, "\RenCounitAlt_{\Tm, \ctx}"] \\
      \Tm(\model)_{\RenCounitAlt_{\Ctx}(\ctxI)}
      \arrow[r, "\map^{\pbMark}"'] &
      \Tm(\model)_{\RenCounitAlt_{\Ctx}(\ctx)}
    \end{tikzcd}
  \end{equation*}
  Suppose that \(\ctx = (\ty_{1}, \dots, \ty_{\nat})\) and
  \(\ctxI = (\tyI_{1}, \dots, \tyI_{\natI})\). Then \(\idxmor\) is a
  map \(\{1, \dots, \natI\} \to \{1, \dots, \nat\}\), and \(\pth\)
  assigns to each \(\idxI \in \{1, \dots, \natI\}\) a homotopy
  \(\pth_{\idxI}\) filling the following diagram.
  \begin{equation*}
    \begin{tikzcd}
      \yoneda \compr{\ty_{\nat}}
      \arrow[rr, "\comprProj^{\nat}_{\idxmor(\idxI)}(\ty_{\nat})"]
      \arrow[d, "\map"'] &
      [4ex] &
      [2ex]
      \yoneda \compr{\ty_{\idxmor(\idxI)}}
      \arrow[d, "\comprProjI(\ty_{\idxmor(\idxI)})"] \\
      \yoneda \compr{\tyI_{\natI}}
      \arrow[r, "\comprProj^{\natI}_{\idxI}(\tyI_{\natI})"'] &
      \yoneda \compr{\tyI_{\idxI}}
      \arrow[r, "\comprProjI(\tyI_{\idxI})"'] &
      \Tm(\model)
    \end{tikzcd}
  \end{equation*}

  We show that \(\RenAlt_{\model}\) is indeed a model of
  \(\basicTT\). The empty sequence is a final object in
  \(\Ctx(\RenAlt_{\model})\). For context extension, let
  \(\ctx = (\ty_{1}, \dots, \ty_{\nat}) \in \Ctx(\RenAlt_{\model})\) be
  an object and let
  \(\ty_{\nat + 1} : \yoneda \ctx \to \Ty(\RenAlt_{\model})\) be a
  section. By definition, \(\ty_{\nat + 1}\) is a section
  \(\yoneda \compr{\ty_{\nat}} \to \Ty(\model)\), and thus
  \((\ty_{1}, \dots, \ty_{\nat + 1})\) is an object in
  \(\Ctx(\RenAlt_{\model})\). The projection
  \(\comprProj(\ty_{\nat + 1}) : \compr{\ty_{\nat + 1}} \to
  \compr{\ty_{\nat}}\) is extended to the morphism
  \(\comprProj'(\ty_{\nat + 1}) = (\comprProj(\ty_{\nat + 1}), \lambda
  \idx. \idx, \lambda \idx. \id) : (\ty_{1}, \dots, \ty_{\nat + 1}) \to
  (\ty_{1}, \dots, \ty_{\nat})\) in \(\Ctx(\RenAlt_{\model})\). The
  element \(\nat + 1 \in \{1, \dots, \nat + 1\}\) determines a section
  \(\yoneda (\ty_{1}, \dots, \ty_{\nat + 1}) \to
  \Tm(\RenAlt_{\model})\), and the diagram
  \begin{equation}
    \label{eq:renaming-model-context-extension}
    \begin{tikzcd}
      \yoneda (\ty_{1}, \dots, \ty_{\nat + 1})
      \arrow[r, "\nat + 1"]
      \arrow[d, "\comprProj'(\ty_{\nat + 1})"'] &
      \Tm(\RenAlt_{\model})
      \arrow[d, "\typeof(\RenAlt_{\model})"] \\
      \yoneda (\ty_{1}, \dots, \ty_{\nat})
      \arrow[r, "\ty_{\nat + 1}"'] &
      \Ty(\RenAlt_{\model})
    \end{tikzcd}
  \end{equation}
  commutes. We show that \cref{eq:renaming-model-context-extension} is
  a pullback. Let
  \(\ctxI = (\tyI_{1}, \dots, \tyI_{\natI}) \in \Ctx(\RenAlt_{\model})\)
  and let \((\map, \idxmor, \pth) : \ctxI \to \ctx\) be a morphism. A
  morphism \(\ctxI \to (\ty_{1}, \dots, \ty_{\nat + 1})\) over
  \(\ctx\) consists of a morphism
  \(\mapI : \compr{\tyI_{\natI}} \to \compr{\ty_{\nat + 1}}\) over
  \(\compr{\ty_{\nat}}\), an element
  \(\idxI \in \{1, \dots, \natI\}\), and a homotopy
  \(\pthI : \comprProjI(\tyI_{\idxI}) \comp
  \comprProj^{\natI}_{\idxI}(\tyI_{\natI}) \simeq \comprProjI(\ty_{\nat +
    1}) \comp \mapI\). On the other hand, a section
  \(\yoneda \ctxI \to \Tm(\RenAlt_{\model})\) over
  \(\ty_{\nat + 1} \comp (\map, \idxmor, \pth)\) consists of an
  element \(\idxI \in \{1, \dots, \natI\}\) and a homotopy
  \(\pthII : \typeof(\model) \comp \comprProjI(\tyI_{\idxI}) \comp
  \comprProj^{\natI}_{\idxI}(\tyI_{\natI}) \simeq \ty_{\nat + 1} \comp
  \map\). The canonical map
  \(\yoneda (\ty_{1}, \dots, \ty_{\nat + 1}) \to \ty_{\nat +
    1}^{\pbMark} \Tm(\RenAlt_{\model})\) takes
  \((\mapI, \idxI, \pthI)\) to \((\idxI, \pthII)\) where \(\pthII\) is
  defined by the following pasting diagram.
  \begin{equation*}
    \begin{tikzcd}
      \yoneda \compr{\tyI_{\natI}}
      \arrow[r, "\comprProj^{\natI}_{\idxI}(\tyI_{\natI})"]
      \arrow[d, "\mapI"']
      \arrow[dd, bend right = 13ex, "\map"']
      \arrow[dr, phantom, "\overset{\pthI}{\simeq}"] &
      [6ex]
      \yoneda \compr{\tyI_{\idxI}}
      \arrow[d, "\comprProjI(\tyI_{\idxI})"] \\
      [2ex]
      \yoneda \compr{\ty_{\nat + 1}}
      \arrow[r, "\comprProjI(\ty_{\nat + 1})"']
      \arrow[d, "\comprProj(\ty_{\nat + 1})"]
      \arrow[dr, pbMark] &
      \Tm(\model)
      \arrow[d, "\typeof(\model)"] \\
      [2ex]
      \yoneda \compr{\ty_{\nat}}
      \arrow[r, "\ty_{\nat + 1}"'] &
      \Ty(\model)
    \end{tikzcd}
  \end{equation*}
  Since the lower square is a pullback, we have
  \begin{math}
    \yoneda (\ty_{1}, \dots, \ty_{\nat + 1}) \simeq
    \ty_{\nat + 1}^{\pbMark}\Tm(\RenAlt_{\model}).
  \end{math}

  By the construction of context extension, \(\RenCounitAlt_{\Ctx}\),
  \(\RenCounitAlt_{\Ty}\), and \(\RenCounitAlt_{\Tm}\) determine a
  morphism \(\RenCounitAlt : \RenAlt_{\model} \to \model\) of models of
  \(\basicTT\). Since
  \(\Ty(\RenAlt_{\model}) \simeq \RenCounitAlt_{\Ctx}^{\pbMark}
  \Ty(\model)\) by definition,
  \(\RenCounitAlt : \RenAlt_{\model} \to \model\) is a renaming model.
\end{construction}

\begin{proof}[Proof of \cref{prop:renaming-model-tm-discrete}]
  Applying \cref{cst:renaming-model-alt} to \(\Ren_{\model}\), we have
  the renaming model \(\RenAlt_{\Ren_{\model}} \to \Ren_{\model}\). By
  the initiality of \(\Ren_{\model}\), it admits a section, and thus
  \(\Ren_{\model}\) is a retract of \(\RenAlt_{\Ren_{\model}}\). Then,
  since \(\Tm(\RenAlt_{\Ren_{\model}})\) is a \(0\)-truncated right
  fibration over \(\Ctx(\RenAlt_{\Ren_{\model}})\) by construction, it
  follows that \(\Tm(\Ren_{\model})\) is a \(0\)-truncated right
  fibration over \(\Ctx(\Ren_{\model})\).
\end{proof}

\begin{remark}
  It is natural to expect that
  \(\Ren_{\model} \simeq \RenAlt_{\model}\), but we do not try to prove
  this because this is not necessary for the coherence results.
\end{remark}

\end{document}